\documentclass[10pt]{amsart}
\usepackage{amsmath,paralist,amsfonts,stmaryrd, amssymb,amscd,amsthm,amsbsy,upref,mathrsfs}
\numberwithin{equation}{section}
\newtheorem{thm}{Theorem}[section]
\newtheorem{lem}[thm]{Lemma}
\newtheorem{cor}[thm]{Corollary}
\newtheorem{prop}[thm]{Proposition}
\theoremstyle{definition}

\newtheorem{defn}[thm]{Definition}
\newtheorem{assump}[thm]{Assumption}

\newtheorem{rem}[thm]{Remark}

\newtheorem*{remark}{Remark}

\newtheorem{nclaim}{Claim}

\newcommand{\ve}{\varepsilon}
\newcommand{\supp}{\mathrm{\rm supp\,}}
\newcommand{\dist}{\mathrm{dist}}
\newcommand{\ball}{\mathrm{\rm ball\,}}
\newcommand{\ran}{\mathrm{\rm ran\,}}
\newcommand{\rank}{\mathrm{\rm rank\,}}
\newcommand{\age}{\mathrm{\rm age\,}}

\newcommand{\weight}{\mathrm{\rm weight\,}}

\newcommand{\im}{\mathrm{\rm im\,}}
\newcommand{\R}{\mathbb R}
\newcommand{\N}{\mathbb N}

\newcommand{\Q}{\mathbb Q}
\newcommand{\half}{{\textstyle\frac12}}

\newcommand{\XK}{\mathfrak X_{\mathrm K}}
\newcommand{\X}{\mathfrak{X}}
\newcommand{\Gammamax}{\Gamma^{\text{max}}}
\DeclareMathOperator{\Ker}{Ker}

\begin{document}

\title [HI $\mathscr{L}_{\infty}$ spaces with few but not very few operators] {Hereditarily indecomposable, separable $\mathscr L_\infty$-spaces with $\ell_1$ dual having few operators but not very few operators}

\maketitle

\newpage
\section{Introduction}
We show that given any $k \in \N, k \geq 2$ there is a separable $\mathscr{L}_{\infty}$ space $\X_k$ which has the following properties:-
\begin{enumerate}
\item $\X_k$ is hereditarily indecomposable (HI) and $\X_k^* = \ell_1$.
\item There is a non-compact bounded linear operator $S \colon \X_k \to \X_k$ on $\X_k$, with $S^j \neq 0$ for $1 \leq j < k$ and, $S^k = 0$. 
\item Moreover, $S^j$ ($ 0 \leq j \leq k-1$) is not a compact perturbation of any linear combination of the operators $S^l, l \neq j$.
\item The operator $S \colon \X_k \to \X_k$ is strictly singular (and consequently $S^j$ is strictly singular for all $j \geq 1$).
\item Whenever $T \colon \X_k \to \X_k$ is a bounded linear operator on $\X_k$, there are $\lambda_i \in \R, (0 \leq i \leq k-1)$ and a compact operator $K \colon \X_k \to \X_k$ with $T = \sum_{i=0}^{k-1} \lambda_i S^i  + K$
\end{enumerate}

We note that as a consequence of (1), (2) and (4), the Calkin algebra $\mathcal{L}(\X_k)/ \mathcal{K}(\X_k) $ is $k$ dimensional with basis $\{ I, S, \dots S^{k-1} \}$. More precisely, it is isomorphic as an algebra to the subalgebra $\mathcal{A}$ of $k\times k$ upper-triangular-Toeplitz matrices, i.e. $\mathcal{A}$ is the subalgebra of $\text{Mat}(k\times k)$  generated by \[
\left\{  \begin{pmatrix} 0 & 1 \\ & 0 & 1 \\  & & \ddots & \ddots \\ & & & \ddots & 1 \\ & & & & 0 \end{pmatrix}^j : 0 \leq j \leq k-1 
\right\}
\]
An explicit isomorphism is given by $\psi \colon \mathcal{L}(\X_k) / \mathcal{K}(\X_k) \to \mathcal{A} \, \cong \R[X]/ \langle x^k \rangle$, \[
\sum_{j=0}^{k-1} \lambda_j S^j \mapsto
 \begin{pmatrix} \lambda_0 & \lambda_1 & \lambda_2 & \cdots & \cdots & \lambda_{k-1} \\

0 & \lambda_0 & \lambda_1 & \lambda_2 & \cdots & \lambda_{k-2} \\

0 & 0 & \lambda_0 & \lambda_1 & \ddots &  \vdots \\

\vdots & \vdots & 0 & \ddots & \ddots& \vdots \\

\vdots & \vdots  & \vdots & & \ddots & \vdots \\

0 & 0 & 0 & \cdots & 0 & \lambda_0

\end{pmatrix}
\]

As a consequence of (3) we see that $\X_k$ has the few operators  property but not the very few operators property. In other words, $\R I + \mathcal{K}(\X_k) \subsetneqq \mathcal{L}(\X_k)  \subseteq \R I + \mathcal{S}\mathcal{S} (\X_k) $ (where $\mathcal{S}\mathcal{S} (\X_k) $ is the space of strictly singular operators on $\X_k$ ). We thus have a negative solution to problem 10.7 of Argyros and Haydon (\cite{AH}). 

We remark also that all operators on the spaces $\X_k$ have non-trivial closed invariant subspaces. Indeed, by a result of Lomonosov (see, eg. \cite{ASm} or \cite{Lom}), if an operator $T$ commutes with a non-zero compact operator, then $T$ has a proper closed invariant subspace. In particular, if there is some polynomial of $T$ which is compact and non-zero, then certainly $T$ has a proper closed invariant subspace.

For an operator $T \colon \X_k \to \X_k$ on $\X_k$, $T = \sum_{j=0}^{k-1} \lambda_j S^j  + K$, we consider the polynomial of $T$, given by $\mathcal{P}(T):= (T - \lambda_0 I)^{k}$. It follows (by the ring isomorphism of the Calkin algebra with the ring $\R[X] / \langle x^k \rangle$) that $\mathcal{P}(T)$ is a compact operator. So if $\mathcal{P}(T) \neq 0$ then we are done by the result of Lomonosov. Otherwise it is cleat that $\lambda$ is an eigenvalue of $T$, so that it has a one dimensional invariant subspace.

\subsection{Acknowledgements} 
The author would like to thank his PhD supervisor, Professor R.G. Haydon for his invaluable suggestions, help and support throughout the writing of this paper.

\section{The Basic Construction}
The fundamental idea is to modify the space $\XK$ constructed in \cite{AH} in order to obtain a space with the desired properties. We will therefore be working with two strictly increasing sequences of natural numbers $(m_j)$ and $(n_j)$ which satisfy the same assumptions as in \cite{AH}. We recall what the precise assumptions on these sequences are.

\begin{assump}\label{mnAssump}
We assume that $(m_j,n_j)_{j\in \N}$ satisfy the following:
\begin{enumerate}
\item $m_1\ge 4$; \item $m_{j+1} \ge m_j^2$;\item $n_1\ge m_1^2$; \item $n_{j+1} \ge
(16n_j)^{\log_2m_{j+1}}=m_{j+1}^2(4n_j)^{\log_2m_{j+1}}.$
\end{enumerate}
\end{assump}

We will construct $\X_k$ using the generalised Bourgain-Delbaen construction described in \cite{AH}. We need the following (slight modification) of theoerm 3.5 appearing in \cite{AH}.

\begin{thm}\label{BDThm}
Let $(\Delta_q)_{q\in \N}$ be a disjoint sequence of non-empty
finite sets; write $\Gamma_q=\bigcup_{1\le p\le
q}\Delta_p$, $\Gamma=\bigcup_{ p\in \N}\Delta_p$. Assume that
there exists $\theta<\frac12$ and a mapping $\tau$ defined on
$\Gamma\setminus \Delta_1$, assigning to each $\gamma\in
\Delta_{q+1}$ a tuple of one of the forms:
\begin{enumerate}
\setcounter{enumi}{-1}
 \item $(\alpha, \xi)$ with $0\le \alpha\le 1$ and $\xi\in \Gamma_q$;
 \item $(p,\beta, b^*)$ with $0\le p< q$, $0<\beta\le \theta$ and $b^*\in \ball
 \ell_1\left(\Gamma_q\setminus \Gamma_p\right)$;
 \item $(\alpha,\xi,p,\beta,b^*)$ with $0<\alpha\le 1$, $1\le p <q$, $\xi\in
 \Gamma_p$, $0<\beta\le \theta$  and $b^*\in \ball
 \ell_1\left  (\Gamma_q\setminus \Gamma_p\right)$.
 \end{enumerate}
Then there exist $d_\gamma^* = e^*_\gamma-c^*_\gamma\in
\ell_1(\Gamma)$ and projections $P^*_{(0,q]}$ on $\ell_1(\Gamma)$
uniquely determined by the following properties:
\begin{enumerate}
  \item[(A)]
  $\displaystyle P^*_{(0,q]}d^*_\gamma = \begin{cases}
 d^*_\gamma\qquad\qquad\qquad\qquad\qquad\text{if }\gamma\in \Gamma_q\\
 0\ \qquad\qquad\qquad\qquad\qquad\text{if }\gamma\in \Gamma\setminus
 \Gamma_q
 \end{cases}$
 \
 \item[(B)] $\displaystyle \qquad\,c^*_\gamma = \begin{cases}
  0  \qquad\qquad\qquad\qquad\qquad\ \text{if }\gamma\in \Delta_1\\
  \alpha e^*_\xi \qquad \qquad\qquad\qquad\quad\text{if } \tau(\gamma) = (\alpha,
  \xi)\\
  \beta (I-P^*_{(0,p]}) b^*\qquad\quad \quad\text{
  if }\tau(\gamma)=(p, \beta, b^*)\\
        \alpha e^*_\xi + \beta(I-P^*_{(0,p]}) b^*\quad\ \text{
 if }\tau(\gamma)=(\alpha, \xi, p,\beta, b^*).
 \end{cases}$
 \end{enumerate}
 The family $(d^*_\gamma)_{\gamma\in \Gamma}$ is a basis for
 $\ell_1(\Gamma)$ with basis constant at most $M=
 (1-2\theta)^{-1}$.  The norm of each projection $P^*_{(0,q]}$ is at
 most $M$.  The biorthogonal vectors $d_\gamma$ generate a
 $\mathscr L_{\infty ,M}$-subspace $X(\Gamma,\tau)$ of $\ell_\infty(\Gamma)$.
 For each $q$ and each $u\in \ell_\infty(\Gamma_q)$, there is a
 unique $i_q(u)\in [d_\gamma:\gamma\in \Gamma_q]$ whose restriction
 to $\Gamma_q$ is $u$; the extension operator
 $i_q:\ell_\infty(\Gamma_q)\to X(\Gamma,\tau)$ has norm at most $M$.
  The subspaces $M_q=[d_\gamma:\gamma\in \Delta_q]
  =i_q[\ell_\infty(\Delta_q)]$ form a
finite-dimensional decomposition (FDD) for $X$; if this FDD is
shrinking then $X^*$ is naturally isomorphic to $\ell_1(\Gamma)$.
\end{thm}
We omit the proof since it is the same as in \cite{AH}. We will also use the notation developed by Argyros and Haydon in \cite{AH}. In particular, as observed in the above theorem, the subspaces $M_n = [d_\gamma:\gamma\in \Delta_n]$
form a finite-dimensional decomposition for $X=X(\Gamma, \tau)$. For each interval
$I\subseteq \N$ we define the projection $P_I:X\to \bigoplus
_{n\in I}M_n$ in the natural way; this is consistent with our use of
$P^*_{(0,n]}$ in Theorem~\ref{BDThm}. As in \cite{AH}, many of the arguments will
involve sequences of vectors that are block sequences with respect
to this FDD. It will therefore be useful to make the following definition; for $x \in X$, we define the {\em range}  of $x$, denoted $\ran x$, to be the smallest interval $I\subseteq \N$ such that $x\in
\bigoplus_{n\in I}M_n$.  

We make one further remark on terminology (taken from \cite{AH}). If $\gamma \in \Delta_{n+1}$, we say $c_{\gamma}^*$ is a \emph{Type 0 BD-functional} if $\tau(\gamma) = (\alpha, \xi)$, a \emph{Type 1 BD-functional} if $\tau(\gamma) = (p, \beta, b^*)$ or a \emph{Type 2 BD-functional} if $\tau(\gamma) = (\alpha, \xi, p, \beta, b^*)$.

Our space will be a specific BD space very similar to the space $\mathfrak X_{\mathrm K}$ constructed in \cite{AH}.  We adopt the same notation used in \cite{AH}, in which elements $\gamma$ of $\Delta_{n+1}$ automatically code the corresponding BD-functionals. Consequently, we can write $X(\Gamma)$ rather than $X(\Gamma, \tau)$ for the resulting $\mathscr{L}_{\infty}$ space. To be more precise, an element $\gamma$ of $\Delta_{n+1}$ will be a tuple
of one of the forms:
\begin{enumerate}
\item $\gamma= (n+1, p, \beta,b^*)$,\quad in which case $\tau(\gamma)=(p,\beta,b^*)$;
\item $\gamma = (n+1,\xi,\beta,b^*)$ in which case $\tau(\gamma)=(1,\xi,\rank \xi, \beta,b^*)$.
\end{enumerate}
In each case, the first co-ordinate of $\gamma$ tells us what the
{\em rank} of $\gamma$ is, that is to say to which set
$\Delta_{n+1}$ it belongs, while the remaining co-ordinates specify
the corresponding BD-functional.

We observe that BD-functionals of Type 0. In the definition of a Type 2 functional,
the scalar $\alpha$ that occurs
 is always 1 and $p$ equals $\rank \xi$. As in the Argyros Haydon construction, we shall make
the further restriction the weight $\beta$ must be of the form
$m_j^{-1}$, where the sequences $(m_j)$ and $(n_j)$ satisfy
Assumption~\ref{mnAssump}.  We shall say that the element $\gamma$
has {\em weight} $m_j^{-1}$. In the case of
a Type 2 element $\gamma=(n+1, \xi, m^{-1}_j,b^*)$ we shall insist
that $\xi$ be of the same weight $m_j^{-1}$ as $\gamma$.

To ensure that the sets $\Delta_{n+1}$ are finite we shall admit
into $\Delta_{n+1}$ only elements of weight $m_j$ with $j\le n+1$. A
further restriction involves the recursively defined function called ``age'' (also defined in \cite{AH}). For a Type 1 element $\gamma=(n+1, p, \beta, b^*)$ we
define $\age\gamma=1$. For a Type 2 element $\gamma=(n+1, \xi,
m^{-1}_j,b^*)$, we define $\age \gamma= 1 + \age \xi$, and further
restrict the elements of $\Delta_{n+1}$ by insisting that the age of
an element of weight $m_j^{-1}$ may not exceed $n_j$. Finally, we
shall restrict the functionals $b^*$ that occur in an element of
$\Delta_{n+1}$ by requiring them to lie in some finite subset $B_n$
of $\ell_1(\Gamma_n)$. It is convenient to fix an increasing
sequence of natural numbers $(N_n)$ and take $B_{p,n}$ to be the set
of all linear combinations $b^*=\sum_{\eta\in \Gamma_n\setminus
\Gamma_p}a_\eta e^*_\eta$, where $\sum_\eta|a_\eta|\le 1$ and each
$a_\eta$ is a rational number with denominator dividing $N_n!$. We
may suppose the $N_n$ are chosen in such a way that $B_{p,n}$ is a
$2^{-n}$-net in the unit ball of $\ell_1(\Gamma_n\setminus
\Gamma_p)$. The above restrictions may be summarized as follows.

\begin{assump}\label{DeltaUpperAssump}
\begin{align*}
\Delta_{n+1} &\subseteq \bigcup_{j=1}^{n+1} \left\{(n+1, p, m_j^{-1},b^*): 0\leq p < n,\ b^*\in B_{p,n}\right\}\\
&\cup
\bigcup_{p=1}^{n-1}\bigcup_{j=1}^{p}\left\{(n+1,\xi,m_j^{-1},b^*):
\xi\in \Delta_p, \weight \xi = m_j^{-1},\ \age\xi<n_j,\ b^*\in
B_{p,n}\right\}
\end{align*}
\end{assump}

As in \cite{AH} we shall also assume that $\Delta_{n+1}$ contains a rich supply of
elements of ``even weight'', more exactly of weight $m_j^{-1}$ with
$j$ even.

\begin{assump}\label{DeltaLowerAssump}
\begin{align*}
\Delta_{n+1} &\supseteq \bigcup_{j=1}^{\lfloor( n+1)/2\rfloor} \left\{(n+1,m_{2j}^{-1},b^*): 0\leq p < n,\ b^*\in B_{p,n}\right\}\\
&\cup\bigcup_{p=1}^{n-1} \bigcup_{j=1}^{\lfloor p/2\rfloor}
\left\{(n+1,\xi,m_{2j}^{-1},b^*): \xi\in \Delta_p, \weight \xi =
m_{2j}^{-1},\ \age\xi<n_{2j},\ b^*\in B_{p,n}\right\}
\end{align*}
\end{assump}

For the main construction, there will be additional restrictions on the
elements with ``odd weight'' $m_{2j-1}^{-1}$, though we will come to these later. To begin with, we shall work with the space $X(\Gamma^{\text{max}} (k))$ where $k \in \N, k \geq 2$ and $\Gamma^{\text{max}}(k)$ is defined by recursion:

\begin{defn}
We define $\Gamma^{\text{max}}(k)$ by the recursion $\Delta_1 = \{ 0, 1, \dots (k-1) \}$,
\begin{align*}
\Delta_{n+1} &=\bigcup_{j=1}^{n+1} \left\{(n+1,p,m_j^{-1},b^*):  0\leq p < n,\ b^*\in B_{p,n}\right\}\\
&\cup \bigcup_{p=1}^{n-1}
\bigcup_{j=1}^p\left\{(n+1,\xi,m_j^{-1},b^*): \xi\in \Delta_p,
\weight \xi = m_j^{-1},\ \age\xi<n_j,\ b^*\in B_{p,n}\right\}
\end{align*}
\end{defn}

\begin{rem} \label{notationBpn}
Later on, we will want to take a suitable subset of $\Gammamax$. To avoid any ambiguity in notation, in the above definition, and throughout the rest of the paper, $B_{p,n}$ will denote the set of all linear combinations $b^* = \sum_{\eta \in \Gammamax_n(k) \setminus \Gammamax_p(k)} a_{\eta}^*$, where, as before, $\sum_{\eta} |a_{\eta}| \leq 1$ and each $a_{\eta}$ is a rational number with denominator dividing $N_n!$. 
\end{rem}

In all that follows, we are assuming we have $k$ points in the set $\Delta_1$ (for some fixed $k \in \N, k \geq 2$). For simplicity of notation, we shall just write $\Gammamax$, $\Gammamax_n$ for $\Gammamax(k)$ (repsectively $\Gammamax_n(k)$). 

Eventually, we want to have a non-compact, bounded linear operator $S$ on our space. To this end, we will need the following theorem. We make use of a single element set which is disjoint from $\Gammamax$, and label the element `undefined'.

\begin{thm} \label{R^*andGConstruction}
There is a map $G: \Gammamax \to \Gammamax \cup \{ \text{undefined} \}$ (we say \emph{$G(\gamma)$ is undefined} if $G(\gamma) = \text{undefined}$, otherwise we say \emph{$G(\gamma)$ is defined}) and a norm 1, linear mapping $R^* \colon \ell_1(\Gammamax) \to \ell_1(\Gammamax)$ satisying:
\begin{enumerate}
\item $G(j) = j-1$ for $1\leq j \leq k-1$ and $G(0)$ is undefined (where we recall $\Delta_1 = \{ 0, 1, \dots k-1 \}$).
\item For elements $\gamma \in \Gammamax\setminus\Delta_1$ such that $G(\gamma)$ is defined, $\rank\gamma = \rank G(\gamma)$ and  $\weight \gamma = \weight G(\gamma)$ (i.e. G preserves weight and rank). Moreover, $\age G(\gamma) \leq \age \gamma$ ($G$ doesn't increase age).
\item \begin{equation*} R^*(e_{\gamma}^*) = \begin{cases} e_{G(\gamma)}^* & \text{ if $G(\gamma)$ is defined} \\ 0 & \text{ otherwise} \end{cases} \end{equation*}
\item \begin{equation*} R^*(d_{\gamma}^*) = \begin{cases} d_{G(\gamma)}^* & \text{ if $G(\gamma)$ is defined} \\ 0 & \text{ otherwise} \end{cases}\end{equation*}
\end{enumerate}
\end{thm}

\begin{proof}
We will construct the maps $G$ and $R^*$ inductively. We note that since $R^*$ will be a linear operator on $\ell_{1}(\Gammamax)$, in order to ensure it is also bounded we only need to be able to control $\|R^*(e_{\gamma}^*)\|$ (for $\gamma\in\Gammamax$). More precisely, if there is some $M \geq 0$ s.t. $\|R^*(e_{\gamma}^*)\| \leq M$ for every $\gamma \in \Gammamax$, then it is elementary to check that $R^*$ is bounded with norm at most $M$. In particular, if property (3) of theorem \ref{R^*andGConstruction} holds, it follows that $\|R^*\| = 1$ 

To begin the inductive constructions of $R^*$ and $G$ we define $G \colon \Delta_{1} \to \Delta_{1}$ by setting $G(j) = j-1$ for $1 \leq j \leq k-1$ and declaring that $G(0)$ is undefined. Noting that $e_{\gamma}^* = d_{\gamma}^*$ for $\gamma \in \Delta_1$, we define $R^*(e_{0}^*) = R^*(d_{0}^*) = 0$ and $R^*(e_{j}^*) = R^*(d_{j}^*) = e_{j-1}^* = d_{j-1}^*$ for $j \geq 1$. We observe that this definition is consistent with the properties (1) - (4) above.

Suppose that we have defined $G: \Gammamax_n \to \Gammamax_n$ and $R^* \colon \ell_1(\Gammamax_n) \to \ell_1(\Gammamax_n)$ satisfying properties (1) - (4) above. We must extend $G$ to $\Gammamax_{n+1}$ and $R^*$ to a map on $\ell_1(\Gammamax_{n+1})$. We consider a $\gamma \in \Delta_{n+1}$ and recall that (see theorem \ref{BDThm}) $e_{\gamma}^* = c_{\gamma}^* + d_{\gamma}^*$. We wish to define $R^* d_{\gamma}^*$ and $R^*e_{\gamma}^*$. By linear independence, we are free to define $R^*d_{\gamma}^*$ however we like. However, since $R^*c_{\gamma}^*$ is already defined ($c_{\gamma}^* \in \ell_{1}(\Gammamax_n))$, and we want $R^*$ to be linear, once we have defined $R^*d_{\gamma}^*$, in order to have linearity we must have $R^*e_{\gamma}^* = R^*c_{\gamma}^* + R^*d_{\gamma}^*$.

Let us consider $R^*c_{\gamma}^*$. We suppose first that $\age\gamma = 1$ so that we can write $\gamma = (n+1, p, \beta, b^*)$ where $b^* \in \ell_{1}(\Gammamax_n \setminus \Gammamax_p)$. Consequently \[
c_{\gamma}^* = \beta\left( I - P^*_{(0,p]}b^*\right) = \beta P^*_{(p,\infty)}b^*
\]
We claim that $R^*P^*_{(p, \infty)}b^* = P^*_{(p,\infty)}R^*b^*$. Indeed, we can write $b^* = \sum_{\delta \, \in \, \Gammamax_n} \alpha_{\delta}d_{\delta}^*$ (for a unique choice of $\alpha_{\delta}$). It follows from property (4) and the inductive hypothesis that \[
R^*P^*_{(p, \infty)}b^* = R^*\big(\sum_{\delta \in \Gammamax_n\setminus\Gammamax_p} \alpha_{\delta}d_{\delta}^*\big	) = \sum_{\substack{ \delta \in \Gammamax_n \setminus \Gammamax_p \cap \\ \{\eta \, \in \, \Gammamax_n \, : \, G(\eta) \text{ is defined} \}} } \alpha_{\delta}d^*_{G(\delta)}
\]
and it is easily checked (by a similar calculation) that we obtain the same expression for $P^*_{(p,\infty)}R^*b^*$. It follows that \[
R^*c_{\gamma}^* = \beta P^*_{(p,\infty)} R^*b^*
\]
We define $G(\gamma)$ by \[
G(\gamma) = \begin{cases} \text{undefined} & \text{ if $P^*_{(p,\infty)} R^*b^* = 0$} \\ (n+1, p, \beta, R^*b^*) & \text{ otherwise} \end{cases}
\]
where it is a simple consequence of the facts that $R^*\colon \ell_1(\Gammamax_n) \to \ell_1(\Gammamax_n)$ has norm 1 and satisfies property (3) that the element $(n+1,p,\beta, R^*b^*) \in \Gammamax$. In the case where $P^*_{(p,\infty)} R^*b^* \neq 0$, $G(\gamma)$ is defined (with the definition as above) and it is evident that $R^*c_{\gamma}^* = c_{G(\gamma}^*)$. We can define $R^*d_{\gamma}^* = d_{G(\gamma)}^*$ and it follows by linearity (and the fact that $e_{\gamma}^* = c_{\gamma}^* + d_{\gamma}^*$) that we have $R^*e_{\gamma}^* = e_{G(\gamma)}^*$ as required. Otherwise, when $P^*_{(p,\infty)} R^*b^* = 0$, $R^*c_{\gamma}^* = 0$, so we can set $R^*d_{\gamma}^* = 0$ and again by linearity we get that $R^*e_{\gamma}^* = 0$.

Now if $\gamma$ has age $>1$, we can write $\gamma = (n+1, \xi, \beta, b^*)$ and $c_{\gamma}^* = e_{\xi}^* + \beta P^*_{(\rank\xi,\infty)}b^*$. Consequently 
\begin{align*}
R^*c_{\gamma}^* &= R^*(e_{\xi}^*) + \beta P^*_{(\rank\xi,\infty)}R^*b^* \\
&= \begin{cases} e_{G(\xi)}^* + \beta P^*_{(\rank\xi,\infty)}R^*b^* & \text{ if $G(\xi)$ is defined} \\ \beta P^*_{(\rank\xi,\infty)}R^*b^* & \text{ otherwise} \end{cases}
\end{align*}
It follows that if $G(\xi)$ is undefined and $P^*_{(\rank\xi,\infty)}R^*b^* = 0$ then $R^*c_{\gamma}^* = 0$. In this case we declare $G(\gamma)$ to be undefined. Otherwise, there are two remaining possiblities
\begin{enumerate}[(i)]
\item $G(\xi)$ is undefined but $P^*_{(\rank\xi,\infty)}R^*b^* \neq 0$. In this case, it is easily verified that the element $G(\gamma) := (n+1, \rank\xi, \beta, R^*b^*) \in \Gammamax$ 
\item $G(\xi)$ is defined. It is again easily checked that the element $G(\gamma):= (n+1, G(\xi), \beta, R^*b^*) \in \Gammamax$ (here we note that in addition to the above arguments, we also need the inductive hypothesis that $G$ does not increase the age of an element).
\end{enumerate}
In either of these cases, we see that $R^*c_{\gamma}^* = c_{G(\gamma)}^*$. We can define $R^*d_{\gamma}^*$ to be $d_{G(\gamma)}^*$ and as before, we then necessarily have $R^*e_{\gamma}^* = e_{G(\gamma)}^*$ (in order that $R^*$ be linear).

We have thus succeeded in extending the maps $G$ and $R^*$. By induction, we therefore obtain maps $G: \Gammamax \to \Gammamax \cup \{ \text{undefined} \}$ and $R^* \colon (c_{00}(\Gammamax), \| \cdot \|_{1}) \to (c_{00}(\Gammamax), \| \cdot \|_{1})$ satisfying the four properties above (here $(c_{00}(\Gammamax), \| \cdot \|_{1})$ is the dense subspace of $\ell_1(\Gammamax)$ consisting of all finitely supported vectors). It follows from property (4) and the argument above that $R^*$ is continuous, linear with $\|R^*\|=1$. It therefore extends (uniquely) to a bounded linear map $R^*\colon\ell_1(\Gammamax) \to \ell_1(\Gammamax)$ also having norm $1$. This completes the proof.
\end{proof}
We make some important observations about the mappings $G$ and $R^*$. 
\begin{lem}\label{DualOfR^*RestrictsProperly}
The dual operator of $R^*$, which we denote by $(R^*)' \colon \ell_1(\Gammamax)^* \to \ell_{1}(\Gammamax)^*$ restricts to give a bounded linear operator $R:= (R^*)'|_{X(\Gammamax)} \colon X(\Gammamax) \to X(\Gammamax)$ of norm at most $1$.
\end{lem}
\begin{proof}
It is a standard result that the dual operator $(R^*)'$ is bounded with the same norm as $R^*$. It follows that the restriction of the domain to $X(\Gammamax)$ is a bounded, linear operator into $\ell_1(\Gammamax)^*$ with norm at most $1$. It only remains to see that this restricted mapping actually maps into $X(\Gammamax)$. Since the family $(d_{\gamma})_{\gamma\in\Gammamax}$ is a basis for $X(\Gammamax)$, it is enough to see that the image of $d_{\gamma}$ under $(R^*)'$ lies in $X(\Gammamax)$. We claim that \[
(R^*)'d_{\delta} = \sum_{\gamma \in G^{-1}(\delta)} d_{\gamma} \]
which would complete the proof. Since $(d_{\gamma}^*)_{\gamma\in\Gammamax}$ is a basis for $\ell_{1}(\Gammamax)$ it is enough to show that for every $\theta \in \Gammamax$ \[
\big((R^*)'d_{\delta}\big)d_{\theta}^* = d_{\delta} (R^*d_{\theta}^*) = \big( \sum_{\gamma \in G^{-1}(\delta)} d_{\gamma} \big) d_{\theta}^* \]

The right hand side of this expression is easy to evaluate; it is only non-zero when $G(\theta) = \delta$, in which case it is equal to $1$. In particular, if $G(\theta)$ is undefined, then the right hand side of the expression is certainly 0, as is $d_{\delta}(R^*d_{\theta}^*)  = d_{\delta}(0)$. If $G(\theta)$ is defined, the left hand side of the expression is $d_{\delta}(d_{G(\theta)}^*)$ which is clearly $1$ if $G(\theta) = \delta$ and $0$ otherwise. So the expressions are indeed equal, as required.
\end{proof}
\begin{lem}\label{G^kalwaysundefined}
For every $\gamma \in \Gammamax$, there is a unique $l, 1\leq l \leq k$ such that $G^j(\gamma)$ is defined whenever $1\leq j < l$ but $G^l(\gamma)$ is undefined.
\end{lem}
\begin{proof}
The uniqueness is easy; if $G(\gamma)$ is defined, $l$ is the maximal $j \in \N$ such that $G^{j-1}(\gamma)$ is defined. Otherwise we must have $l=1$. So we only have to prove existence of such $l$.

We prove by induction on $n$ that if $\rank\gamma = n$ there is some $1\leq l \leq k$ such that $G^{l}(\gamma)$ is undefined, $G^j(\gamma)$ is defined if $j < l$. The case where $n=1$ is clear from the construction of the map $G$. So, inductively, we assume the statement holds whenever $k \leq n$. Let $\gamma \in \Delta_{n+1}$ and consider 2 cases.
\begin{enumerate}[(i)]
\item $\age\gamma = 1$. We write $\gamma = (n+1,p, \beta, b^*)$. Now $b^* \in \ell_1(\Gammamax_n)$ and by the inductive hypothesis, for every $\theta \in \Gammamax$ with $\rank\theta \leq n$, there is some $l \leq k$ such that $G^l(\theta)$ is undefined. It follows that we must have $(R^*)^lb^* = 0$ for some $1\leq l \leq k$. So it is certainly true that $P^*_{(p,\infty)}(R^*)^lb^* = 0$ for some $1 \leq l \leq k$. It follows by construction of the map $G$ that the `$l$' we seek is the minimal $l$ ($1\leq l \leq k$) such that $P^*_{(p,\infty)}(R^*)^lb^* = 0$.
\item $\age\gamma > 1$. We write $\gamma = (n+1, \xi, \beta, b^*)$. If $G(\gamma)$ is undefined we are done; we must have $l=1$. Otherwise we have either $G(\gamma) = (n+1, \rank\xi, \beta, R^*b^*)$ or $G(\gamma) = (n+1, G(\xi), \beta, R^*b^*)$. In the first of these two possibilities, the same argument as in the previous case shows that the $l$ we seek is the minimal $l$ ($2 \leq l \leq k$) such that $P^*_{(p,\infty)}(R^*)^lb^* = 0$. In the latter case, $G(\xi)$ is defined. But, since $\rank \xi = \rank G(\xi) < n$, we know by the inductive hypothesis that $G^{l_0}(\xi)$ is undefined for some $2\leq l_0 \leq k$ and $G^j(\xi)$ is defined for $j<l_0$. Now, if $P^*_{(\rank\xi, \infty)} (R^*)^{l_0}b^* = 0$, then it follows from construction of $G$ that $G^{l_0}(\gamma)$ is undefined and $G^j(\gamma)$ is defined for $j < l_0$ so we are done. Otherwise, it follows from an argument above that $l_0 < k$, and there is some (minimal) $l$, $l_0 < l \leq k$ with $P^*_{(\rank\xi, \infty)} (R^*)^l b^* = 0$. Once again, this is the desired $l$.
\end{enumerate}
\end{proof}
\begin{cor}\label{R^k=0}
The maps $R^* \colon \ell_1(\Gammamax) \to \ell_1(\Gammamax)$ and $R \colon X(\Gammamax) \to X(\Gammamax)$ satisfy $(R^*)^k = 0$ and $R^k = 0$.
\end{cor}
\begin{proof}
It is clear from lemma \ref{G^kalwaysundefined} that the restriction of $(R^*)^k$ to $c_{00}(\Gammamax)$ is the zero map. It follows by density and continuity that $(R^*)^k=0$. The other claims are immediate from the definition of $R$ as the restriction of the dual operator of $R^*$.
\end{proof} 

To obtain the extra constraints that we place on  ``odd-weight'' elements we will need a function $\sigma \colon \Gamma^{\text{max}} \to \N$ which satisfies
\begin{enumerate}[(1)]
\item $\sigma$ is injective
\item $\sigma(\gamma) > \rank\gamma \quad \forall \, \gamma \in \Gamma^{\text{max}}$
\item for $\gamma \in \Delta_{n+1}$ (i.e. $\rank\gamma = n+1$), $\sigma(\gamma) > \max \{ \sigma(\xi) : \xi \in \Gamma_n^{\text{max}} \} $
\end{enumerate}
Such a $\sigma$ can be constructed recursively as $\Gamma^{\max}$ is constructed. Now, for each $\gamma \in \Gamma^{\text{max}}$ we can well-define a finite set $\Sigma(\gamma)$ by \[
\Sigma(\gamma) := \{ \sigma (\gamma) \} \cup \bigcup_{j=1}^{k-1} \bigcup_{\delta\in G^{-j} (\gamma) } \{ \sigma(\delta) \}
\]
where $G^{-j}(\gamma) := \{ \theta \in \Gammamax : G^{j}(\theta) = \gamma \}$ (so in particular $G^l(\gamma) \in \Gammamax$ for every $l \leq j$).
We have the following lemma

\begin{lem}\label{SigmaGammaContainedinSigmaGGamma}
If $\gamma \in \Gammamax$ is such that $G(\gamma)$ is defined then $\Sigma (\gamma) \subseteq \Sigma (G(\gamma))$
\end{lem}
\begin{proof}
Certainly $\gamma \in G^{-1} (G(\gamma))$ so $\sigma(\gamma) \in \Sigma (G(\gamma))$. Suppose $\sigma(\delta) \in \Sigma(\gamma), \delta \neq \sigma$. So, there is some $ 1 \leq j \leq k-1$ with $\delta \in G^{-j}(\gamma)$, i.e. there is some $\delta$ such that $G^{j}(\delta) = \gamma$. Since $G(\gamma)$ is defined, we must have $G^{j+1}(\delta) = G(\gamma) \in \Gammamax$. In particular, by lemma \ref{G^kalwaysundefined}, we must in fact have had $j < k-1$ so that $j+1 \leq k-1$. Thus $\delta \in G^{-(j+1)} (G(\gamma))$ and $\sigma(\delta) \in \Sigma (G(\gamma))$, as required.
\end{proof}

Before giving our main construction, we document two more observations.

\begin{lem} \label{MonotonicityOfSigmaSets}
If $\gamma, \gamma' \in \Gammamax$, $\rank \gamma > \rank\gamma'$ then $\Sigma(\gamma) > \Sigma(\gamma')$, i.e. $\max \{ k : k \in \Sigma(\gamma') \} < \min \{ k : k \in \Sigma(\gamma) \}$
\end{lem}
\begin{proof}
This follows immediately from the definition of the sets $\Sigma (\gamma)$ and $\Sigma(\gamma')$, the fact that $G$ is rank preserving and the assumption that for $\gamma \in \Delta_{n+1}$, $\sigma(\gamma) > \max \{ \sigma(\xi) : \xi \in \Gammamax_n \}$.
\end{proof}

\begin{lem} \label{IntersectionPropertyofSigmaSets}
Suppose $\gamma, \delta \in \Gammamax$. If $\sigma(\gamma) \in \Sigma(\delta)$ then either $\sigma = \delta$ or there is some $1 \leq j \leq k-1$ such that $G^j(\gamma) = \delta$.
\end{lem}
\begin{proof}
Since $\sigma(\gamma) \in \Sigma(\delta)$ there are two possibilities. Either $\sigma(\gamma) = \sigma(\delta)$ or there is some $1\leq j \leq k-1$ and $\theta \in \Gammamax$ with $G^j(\theta) = \delta$ and $\sigma(\gamma) = \sigma(\theta)$. By injectivity of $\sigma$, this implies that either $\gamma = \delta$, or that $\theta = \gamma$ and $G^j(\gamma) = \delta$ as requied.
\end{proof}

We are finally in a position to describe the main construction. We will take a subset $\Gamma(k) \subset \Gamma^{\text{max}} (=\Gammamax(k))$ by placing some restrictions on the elements of odd weight we permit. (Again, we drop the dependence on '$k$' and just write $\Gamma$ for $\Gamma(k)$). As a consequence of imposing these additional odd weight restrictions, we are also forced to (roughly speaking) also remove those elements $(n+1, p, \beta, b^*)$ and $(n+1, \xi, \beta, b^*)$ of $\Gammamax$ for which the support of $b^*$ is not contained in $\Gamma$, in order that we can apply the Bourgain-Delbaen construction to obtain a space $X(\Gamma)$. Note that the subset $\Gamma$ will also be constructed inductively, so there is no circular argument here. We will denote by $\Delta_n'$ the set of all elements in $\Gamma$ having rank $n$, and denote by $\Gamma_n$ the union $\Gamma_n = \cup_{j\leq n} \Delta_j'$. The permissible elements of odd weight will be as follows. For an age 1 element of odd weight, $\gamma = (n+1, p, m_{2j - 1}^{-1}, b^*)$ we insist that either $b^* = 0$ or $b^* = e_{\eta}^*$ where $\eta \in \Gamma_n \setminus \Gamma_p$ and $\weight\eta = m_{4i}^{-1} < n_{2j-1}^{-2}$. 
For an odd weight element of age $> 1$, $\gamma = (n+1, m_{2j - 1}^{-1}, \xi, b^*)$ we insist that either $b^* = 0$ or $b^* = e_{\eta}^*$ where $\eta \in \Gamma_n \setminus \Gamma_{\rank\xi}$ and $\weight\eta = m_{4k}^{-1} < n_{2j-1}^{-2}$, $k \in \Sigma(\xi)$.Let us be more precise:

\begin{defn} \label{DefnOfGammaAndSpace}

We define recursively sets $\Delta_{n}' \subseteq \Delta_n$. Then $\Gamma_n := \cup_{j \leq n} \Delta'_j$ and $\Gamma := \cup_{n\in\N}\Delta_n' \subseteq \Gammamax$. To begin the recursion, we set $\Delta_1' = \Delta_1$. Then

\begin{align*}
\Delta_{n+1}' &= \bigcup_{j=1}^{\lfloor(n+1)/2\rfloor}
\left\{(n+1,p, m_{2j}^{-1},b^*): 0\leq p<n,\ b^*\in B_{p,n}\cap \ell_1(\Gamma_n) \right\}\\ &\cup
\bigcup_{p=1}^{n-1}\bigcup_{j=1}^{\lfloor
p/2\rfloor}\left\{(n+1,\xi,m_{2j}^{-1},b^*): \xi\in \Delta_p',
\weight \xi = m_{2j}^{-1},\ \age\xi<n_{2j},\
b^*\in B_{p,n}\cap \ell_1(\Gamma_n \setminus \Gamma_p \right\}\\
&\cup\bigcup_{j=1}^{\lfloor(n+2)/2\rfloor}
\left\{(n+1,m_{2j-1}^{-1},b^*):
b^*=0 \text{ or } b^*=e_{\eta}^* \text{ with } \eta\in \Gamma_n \text{ and } \weight \eta= m_{4i}^{-1}<n_{2j-1}^{-2}\right\}\\
&\cup  \bigcup_{p=1}^{n-1}\bigcup_{j=1}^{\lfloor(
p+1)/2\rfloor}\left\{(n+1,\xi,m_{2j-1}^{-1},b^*): \xi\in
\Delta_p', \weight \xi = m_{2j-1}^{-1},\ \age\xi<n_{2j-1} \right.,\\  &\left.
\qquad\qquad\qquad\qquad b^*=0 \text{ or } b^*=e_{\eta}^* \text{ with } \eta\in\Gamma_n\setminus\Gamma_p,\ \weight\eta=
m_{4k}^{-1}<n_{2j-1}^{-2},\ k \in \Sigma(\xi) \right\}.
\end{align*}
Here the $B_{p,n}$ are defined as in \ref{notationBpn}. We define $\X_k$ to be the Bourgain-Delbaen space $X(\Gamma)$ where $\Gamma$ is the subset of $\Gammamax$ just defined.
\end{defn}

For the rest of the paper we will work with the space $\X_k$. We will also drop the prime from the sets $\Delta_{n+1}'$ defined in \ref{DefnOfGammaAndSpace} as we will be working with the set $\Gamma$. As in \cite{AH}, the structure of the space $\X_k$ is most easily
understood in terms of the basis $(d_\gamma)_{\gamma\in\Gamma}$ and the biorthogonal
functionals $d_\gamma^*$. However, we will need to work with the evaluation functionals
$e^*_\gamma$ in order to estimate norms.  To this end, we have the following proposition.

\begin{prop} \label{EvalAnal}Let $n$ be a positive integer and
let $\gamma$ be an element of $\Delta_{n+1}$ of weight $m_j^{-1}$
and age $a \le n_j$. Then there exist natural numbers
$p_0<p_1<\cdots<p_a=n+1$, elements $\xi_1,\dots,\xi_a=\gamma$ of
weight $m_j^{-1}$ with $\xi_r\in \Delta_{p_r}$ and functionals
$b^*_r\in \ball\ell_1\left(\Gamma_{p_r-1}\setminus
\Gamma_{p_{r-1}}\right)$ such that
 \begin{align*} e^*_\gamma &= \sum_{r=1}^a d^*_{\xi_r} +
m_j^{-1}\sum_{r=1}^a P^*_{(p_{r-1},\infty)}b^*_r\\
&=\sum_{r=1}^a d^*_{\xi_r} + m_j^{-1}\sum_{r=1}^a
P^*_{(p_{r-1},p_r)}b^*_r.
\end{align*}
If $1\le t<a$ we have
$$
e_\gamma^*= e^*_{\xi_t} + \sum_{r=t+1}^ad^*_{\xi_r} + m_j^{-1}
\sum_{r=t+1}^a P^*_{(p_{r-1},\infty)}b^*_r\,.
$$
\end{prop}
\begin{proof}
The proof is an easy induction on the age $a$ of $\gamma$. We omit the details because the argument is the same as in \cite{AH} except our $p_0$ is not necessarily 0.
\end{proof}
\begin{rem}
As in \cite{AH}, we shall refer to any of the above identities as the evaluation analysis of the element $\gamma$, and the data $(p_0, (p_r, b_r^*, \xi_r)_{1\leq r \leq a} )$ as the analysis of $\gamma$. We will omit the $p_0$ when $p_0 = 0$.
\end{rem}

We will now construct the operator $S: \X_k \to \X_k$. We need:
\begin{prop} \label{S^*construction}
$\Gamma$ is invariant under $G$. More precisely, if $\, \gamma \in \Gamma \subseteq \Gammamax$ and $G(\gamma)$ is defined, then $G(\gamma) \in \Gamma$. It follows that the map $G\colon \Gammamax \to \Gammamax\cup\{\text{undefined}\}$ defined in theorem \ref{R^*andGConstruction} restricts to give $F \colon \Gamma \to \Gamma\cup\{\text{undefined}\}$. Consequently the map $R^* \colon \ell_1(\Gammamax) \to \ell_1(\Gammamax)$ (also defined in \ref{R^*andGConstruction}) can be restricted to the subspace $\ell_1(\Gamma) \subseteq \ell_1(\Gammamax)$ giving $S^*\colon\ell_1(\Gamma) \to \ell_1(\Gamma)$. $S^*$ is a bounded linear map on $\ell_1(\gamma)$ of norm 1 which satisfies 
\[
S^*e_{\gamma}^* = \begin{cases} 0 & \text{ if $F(\gamma)$ is undefined} \\ e_{F(\gamma)}^* & \text{ otherwise} \end{cases}
\]
for every $\gamma \in \Gamma$. Moreover, the dual operator of $S^*$ restricts to $\X_k$ to give a bounded linear operator $S \colon \X_k \to \X_k$ of norm at most 1, satisfying $S^j \neq 0$ for $1\leq j \leq k-1$, $S^k = 0$.
\end{prop}
\begin{proof}
It is enough to show that when $\gamma \in \Gamma$ and $G(\gamma)$ is defined, then $G(\gamma) \in \Gamma$. The claims about the operator $S^*$ follow immediately from the definition of $S^*$ as the restriction of $R^*$ and the definition of $R^*$. The fact that $S\colon \X \to \X$ is well defined follows by the same argument as in \ref{DualOfR^*RestrictsProperly}; indeed it is seen that for $\delta \in \Gamma$, \[
Sd_{\delta} = \sum_{\gamma \in F^{-1}(\delta)}d_{\gamma}
\]
Moreover, by corollary \ref{R^k=0}, we see that $(S^*)^k= 0$ and therefore that $S^k=0$. That $S^j \neq 0$ for $1\leq j \leq k-1$ is clear from the above formula and consideration of the elements $d_{m}$ for $m \in \Delta_1$.

We use induction on the rank of $\gamma$ to prove that if $\gamma \in \Gamma$ and $G(\gamma)$ is defined, then $G(\gamma) \in \Gamma$. This is certainly true when $\rank\gamma = 1$. Suppose by induction, that whenever $\gamma \in \Gamma$, $\rank\gamma \leq n$ and $G(\gamma)$ is defined, $G(\gamma) \in \Gamma$ and consider a $\gamma \in \Gamma$, of rank $n+1$ such that $G(\gamma)$ is defined. Let us suppose first that this $\gamma$ has age 1. We can write $\gamma = (n+1, p, \beta, b^*)$ where $\supp b^* \subseteq \Gamma_n \setminus \Gamma_p$, and we write $b^* = \sum_{\eta \in \Gamma_n \setminus \Gamma_p} a_{\eta}e_{\eta}^*$. Since $G(\gamma)$ is defined, we have $G(\gamma) =  (n+1, p, \beta, S^*b^*)$. We consider further sub-cases. If $\beta = m_{2j}^{-1}$ for some $j$ (i.e. $\gamma$ is an ``even weight'' element), we see that the only way $G(\gamma)$ fails to be in $\Gamma$ is if $\supp S^*b^* \nsubseteq \Gamma_n \setminus \Gamma_p$. But \[
S^*b^* = \sum_{\substack { \eta \in \Gamma_n\setminus\Gamma_p \, \cap \\ \{\eta  \, : \, F(\eta) \text{ is defined} \} } } a_{\eta}e_{F(\eta)}^*
\]
and by the inductive hypothesis all the $F(\eta)$ in this sum are in $\Gamma_n$. So $\supp S^*b^* \subseteq \Gamma_n \setminus \Gamma_p$ and $G(\gamma) = F(\gamma) \in \Gamma$ as required. In the case where $\beta = m_{2j-1}^{-1}$ (i.e. $\gamma$ is an ``odd-weight'' element) we must also check that the odd weight element $G(\gamma)$ is of a permissible form. Since $G(\gamma)$ is defined, in particular we must have that $P^*_{(p,\infty)} S^*b^* \neq 0$. This of course implies that $S^*b^* \neq 0$ and $b^* \neq 0$. As $\gamma \in \Gamma$, $b^* = e_{\eta}^*$ for some $\eta \in \Gamma_n \setminus \Gamma_p$ where $\weight\eta = m_{4i}^{-1} < n_{2j-1}^{-2}$. Since $S^*b^* \neq 0$, we must have $S^*e_{\eta}^* = e_{F(\eta)}^*$ where in particular, $F(\eta)$ is defined and lies in $\Gamma$ by the inductive hypothesis. Since $G$ is weight preserving, we have $\weight F(\eta) = \weight\eta = m_{4i}^{-1} < n_{2j-1}^{-2}$. Moreover, since $G$ preserves rank, we can now conclude that $F(\gamma) = G(\gamma) = (n+1, p, m_{2j-1}^{-1}, S^*b^*) = (n+1, p, m_{2j-1}^{-1}, e_{F(\eta)}^*) \in \Gamma$ as required.

When $\age\gamma > 1$, we can write $\gamma = (n+1, \beta, \xi, b^*)$ where (in particular) $\supp b^* \subseteq \Gamma_n$ and $\xi \in \Gamma_p$ ($p<n)$. If this $\gamma$ is of even weight, it follows easily from the inductive hyptohesis and arguments similar to the `age 1' case that $G(\gamma) = F(\gamma) \in \Gamma$ (when $G(\gamma)$ is defined). So we consider only the case when $\weight \gamma = \beta = m_{2j-1}^{-1}$ (for some $j$) and $G(\gamma)$ is defined. If $G(\xi) = F(\xi)$ is undefined, then since $G(\gamma)$ is defined, we must have $G(\gamma) = (n+1, \rank\xi, m_{2j-1}^{-1}, S^*b^*)$ and $P^*_{(\rank\xi, \infty)} S^*b^* \neq 0$. So in particular, $b^* \neq 0$ and $S^*b^* \neq 0$. Again by the restrictions on elements of odd weight, we must have $b^* = e_{\eta}^*$ where $\weight\eta = m_{4k}^{-1} < n_{2j-1}^{-2}$, some $k \in \Sigma(\xi)$. Since $S^*b^* \neq 0$, we must have $G(\eta) = F(\eta)$ defined and $S^*b^* = S^*e_{\eta}^* = e_{F(\eta)}^*$. Since $G$ preserves rank, we see as a consequence of the inductive hypothesis, that $\rank F(\eta) \in \Gamma_n \setminus \Gamma_{\rank\xi}$. Furthermore, $\weight F(\eta) = \weight \eta = m_{4k}^{-1} < n_{2j-1}^{-2}$. We conclude that $G(\gamma) = F(\gamma) \in \Gamma$ as required.

In the case where $F(\xi)$ is defined, $G(\gamma) = (n+1, F(\xi), m_{2j-1}^{-1}, S^*b^*)$. If $S^*b^* = 0$ then by the inductive hypothesis, we certainly have $G(\gamma) = F(\gamma) \in \Gamma$ and we are done. Otherwise, we again must have $b^* = e_{\eta}^*$ where $\weight\eta = m_{4k}^{-1} < n_{2j-1}^{-2}$, some $k \in \Sigma(\xi)$ and $S^*b^* = e_{F(\eta)}^*$. Now, $\weight F(\eta) = \weight\eta = m_{4k}^{-1} < n_{2j-1}^{-2}$ and $k \in \Sigma(\xi) \subseteq \Sigma (G(\xi))$ by lemma \ref{SigmaGammaContainedinSigmaGGamma}. 
\end{proof}

Later, we will need the following lemma about the elements of odd weight in $\Gamma$. 
\begin{lem} \label{monotonicity_of_odd_weights}
Let $\gamma \in \Gamma$ be an element of odd weight and $\age\gamma > 1$. Let $\big( p_0, (p_i , \xi_i, b_i^* ) \big)$ be the analysis of $\gamma$, where we know each $b_i^*$ is either $0$ or $e_{\eta_i}^*$ for some suitable $\eta_i$. If there are $i, j$, $1 \leq i < j \leq \age\gamma := a$ with $b_i^* = e_{\eta_i}^*$ and $b_j^* = e_{\eta_j}^*$ then $\weight\eta_j \lneq \weight\eta_i$.
\end{lem}
\begin{proof}
We consider the cases when $i = 1$ and $i > 1$ separately. Suppose $b_1^* = e_{\eta_1}^*$ where $\weight\eta_1 = m_{4i}^{-1}$. By construction of $\Gamma^{\text{max}}$, elements of rank $p$ are only allowed to have weights $m_j^{-1}$ where $1 \leq j \leq p$. So $\weight\eta_1 = m_{4i}^{-1} \implies 4i \leq \rank\eta_1 < \rank\xi_1 = p_1$. By the strict monotonicity of the sequence $m_j$, we get that $m_{\rank\xi_1}^{-1} < m_{4i}^{-1} = \weight\eta_1$.

Now for any $j > 1$, if $b_j^* \neq 0$, $b_j^* = e_{\eta_j}^*$ where $\weight\eta_j = m_{4k}^{-1}$ some $k \in \Sigma (\xi_{j-1})$. Since the mapping $F$ preserves rank of elements, and $\sigma(\theta) > \rank\theta$ $\forall \theta \in \Gamma$ (assumption (2) of the $\sigma$ mapping), it is immediate from the definition of $\Sigma(\xi_{j-1})$ that $k \in \Sigma(\xi_{j-1}) \implies k > \rank\xi_{j-1} \geq \rank\xi_1$. Now \[
\weight\eta_j = m_{4k}^{-1} < m_{4\rank\xi_{j-1}}^{-1} \leq m_{4\rank\xi_1}^{-1} \leq \weight\eta_1
\]
thus concluding the proof for the case $i=1$. The case when $i > 1$ is easy. In this case we suppose $b_i^* = e_{\eta_i}^* $, $\weight\eta_i = m_{4l}^{-1}$ for some $l \in \Sigma(\xi_{i-1})$. Now $j > i$, so (by lemma \ref{MonotonicityOfSigmaSets}) $\Sigma (\xi_{j-1}) > \Sigma (\xi_{i-1})$ so that $l < k$ and therefore $\weight\eta_j = m_{4k}^{-1} < m_{4l}^{-1} = \weight\eta_i$, completing the proof.
\end{proof}

\section{Rapidly Increasing Sequences and the operator $S\colon\X\to\X$}
We recall from \cite{AH} that special classes of block sequences, namely the \emph{rapidly increasing sequences} admit good upper estimates. This class of block sequences will also be useful in our construction. We recall the definition:
\begin{defn}\label{RISDef}
Let $I$ be an interval in $\N$ and let $(x_k)_{k\in I}$ be a
block sequence (with respect to the FDD $(M_n)$). We say that
$(x_k)$ is a {\em rapidly increasing sequence}, or RIS, if there
exists a constant $C$ such that the following hold:
\begin{enumerate}
\item
$\|x_k\|\le C$ for all $k\in \N$,
\end{enumerate} and there is an increasing sequence $(j_k)$ such
that, for all $k$,
\begin{enumerate}\setcounter{enumi}{1}
\item $j_{k+1} > \max\, \text{ran }\,x_k$\,,
\item $|x_k(\gamma)| \le Cm_i^{-1}$ whenever $\weight\gamma = m_i^{-1}$ and
$i<j_k$\,.
\end{enumerate}
If we need to be specific about the constant, we shall refer to a
sequence satisfying the above conditions as a $C$-RIS.
\end{defn}

\begin{rem} \label{SofRISisRIS}
We make the following important observation. If $(x_i)_{i\in\N}$ is a $C$-RIS, then so also is the sequence $(Sx_i)$. We omit the very easy proof.
\end{rem}

We also note that the estimates of the lemmas and propositions 5.2 - 5.6 and 5.8 of \cite{AH} all still hold. The same proofs go through, with only minor modifications to take account of the fact that $p_0$ doesn't need to be $0$ in the evaluation analysis of an element $\gamma$ in our $\Gamma$ (see proposition \ref{EvalAnal}). For convenience, we state proposition 5.6 of \cite{AH} as we shall be making use of it in this paper:- 

\begin{prop}\label{AHProp5.6}
Let $(x_k)_{k=1}^{n_{j_0}}$ be a $C$-RIS.  Then
\begin{enumerate}
\item
For every $\gamma\in \Gamma$ with $\weight \gamma=m_h^{-1}$ we have
$$
|n_{j_0}^{-1}\sum_{k=1}^{n_{j_0}} x_k(\gamma)|\le \begin{cases}
16Cm_{j_0}^{-1}m_h^{-1} &\text{ if }h<j_0\\
4Cn_{j_0}^{-1} + 6Cm_h^{-1} &\text{ if } h\ge j_0 \end{cases}
$$
In particular,
$$
|n_{j_0}^{-1}\sum_{k=1}^{n_{j_0}} x_k(\gamma)|\le 10Cm_{j_0}^{-2},
$$
if $h>j_0$ and
$$
 \|n_{j_0}^{-1}\sum_{k=1}^{n_{j_0}}x_k\|\le 10Cm_{j_0}^{-1}.
$$
\item
If $\lambda_k$ ($1\le k\le n_{j_0}$) are scalars with
$|\lambda_k|\le 1$ and having the  property that
$$
|\sum_{k\in J} \lambda_kx_k(\gamma)|\le C\max_{k\in J}|\lambda_k|,
$$
for every $\gamma$ of weight $m_{j_0}^{-1}$ and every interval
$J\subseteq\{1,2,\dots,n_{j_0}\}$, then
$$
 \|n_{j_0}^{-1}\sum_{k=1}^{j_0}\lambda_kx_k\|\le 10Cm_{j_0}^{-2}.
$$
\end{enumerate}
\end{prop}

Another result of particular importance to us will be the following proposition of \cite{AH}:

\begin{prop}\label{RIStoBlock}
Let $Y$ be any Banach space and $T:\X_k \to Y$ be a bounded
linear operator.  If $\|T(x_k)\|\to0$ for every RIS $(x_k)_{k\in
\N}$ in $\X_k$  then $\|T(x_k)\|\to 0$ for every bounded
block sequence sequence in $\X_k$.
\end{prop}

Again, the same proof as in \cite{AH} goes through. Consequently, we also get (by the same arguments as in \cite{AH})
\begin{prop}\label{Xhasl1Dual}
The dual of $\X_k$ is $\ell_1(\Gamma)$. More precisely the map \[ 
\varphi \colon \ell_{1}(\Gamma) \to \X_k^*
\]
defined by
\[
 \varphi(x^*)x := \langle x , x^* \rangle
 \]
(where $x \in \X_k \subseteq \ell_{\infty}(\Gamma) = \ell_{1}^*(\Gamma)$, $x^* \in \ell_1(\Gamma)$) is an isomorphism.
\end{prop}

This canonical identification of $\X_k^*$ with $\ell_1(\Gamma)$ allows us to prove some important properties of the operator $S\colon \X_k\to\X_k$. 

\begin{lem} \label{imSClosed} 
The dual of the operator $S: \X_k \to \X_k$ is precisely the operator $S^* : \ell_1(\Gamma) \to \ell_1(\Gamma)$ under the canonical identification, $\varphi$, of $\X_k^*$ with $\ell_1(\Gamma)$. Moreover, $S^j$ has closed range for every $1\leq j \leq k-1$, and consequently, $S^j: \X_k \to \im S^j$ is a quotient operator.
\end{lem}

\begin{proof}
We will temporarily denote the dual map of $S$ by $S' : \X_k^*  \to \X_k^*$ so as to not confuse it with the $S^*$ mapping on $\ell_{1}(\Gamma)$. By continuity and linearity, the maps $\varphi ^{-1} S' \varphi$ and $S^*$ are completely determined by their action on the vectors $e_{\gamma}^*$ for $\gamma \in \Gamma$. For $x \in \X_k$, \[
\left( S' \varphi (e_{\gamma}^*)\right) (x) = \varphi (e_{\gamma}^*) Sx = \langle e_{\gamma}^* , Sx \rangle = \langle S^*e_{\gamma}^* , x \rangle = \begin{cases} x(F(\gamma) & \text{ if $F(\gamma)$ is defined} \\ 0 & \text{ otherwise} \end{cases}
\]
It follows from this that 
\begin{align*}
\varphi^{-1} S' \varphi (e_{\gamma}^*) &= \begin{cases} e_{F(\gamma)}^* & \text{ if $F(\gamma)$ is defined} \\ 0 & \text{ otherwise} \end{cases} \\[2pt]
&= S^*(e_{\gamma}^*)
\end{align*}
as required.

We recall now Banach's Closed Range Theorem. A particular consequence of this result is that a bounded linear operator has closed range if and only if its dual operator has closed range. So to see that the image of $S^j$ is closed, it will be enough to see that the image of $(S^*)^j$ is closed. But it is easily seen that the image of $(S^*)^j$ is just $\ell_1(\Gamma \cap \im F^j) \subseteq \ell_{1}(\Gamma)$ (with the obvious embedding), so certainly $(S^*)^j$ has closed image. 

Since $\im S^j$ is closed, it follows immediately that $S^j \colon \X_k \to \im S^j$ is a quotient operator.

\end{proof}

\begin{cor} \label{corSisQuotientOp}
For $1 \leq j \leq k-1$there are $a_j, b_j \in \R$, $a_j, b_j > 0$ such that whenever $x \in \X_k$, \[
a_j\|S^jx\| \leq \dist (x , \Ker S^j ) \leq b_j \| S^jx \|
\]
\end{cor}

\begin{proof}
This is immediate from the fact that $S^j: \X \to \im S^j$ is a quotient operator.
\end{proof}

\begin{cor}
Suppose $\lambda_i \in \R$ ($0\leq i \leq k-1$) are such that $\sum_{i=0}^{k-1} \lambda_i S^i$ is compact. Then $\lambda_i = 0$ for every $i$. Consequently, there does not exist  $j$ ($1\leq j \leq k-1$) such that $S^j$ is a compact perturbation of some linear combination of the operators $S^l, l \neq j$, and $\{ I, S^j : 1\leq j\leq k-1\} $ is a linearly independent set in $\mathcal{L}(\X_k)/ \mathcal{K}(\X_k)$.
\end{cor}

\begin{proof}
It follows by a standard result that $\sum_{i=0}^{k-1} \lambda_i S^i$ is compact if, and only if, the dual operator is compact. But the dual operator is just $T: = \sum_{i=0}^{k-1} \lambda_i (S^*)^i \colon \ell_{1}(\Gamma) \to \ell_{1}(\Gamma)$ and it is now easily seen that this is not compact unless all the $\lambda_i$ are $0$. Indeed, suppose $T$ is compact and consider first the sequence $(e_{\gamma^0_n}^*)_{n=1}^{\infty} \subseteq B_{\ell_1(\Gamma)}$ where $\gamma^0_n = (n+1, m_2^{-1}, 0, e_{0}^*) \in \Gamma$. Then, for $m\neq n$ (observing that $F(\gamma^0_n)$ is undefined for every $n$) we have \[
\| T e_{\gamma^0_n}^* - T e_{\gamma^0_m}^* \|_1 = |\lambda_0 | \| e_{\gamma^0_n}^* - e_{\gamma^0_m}^* \|_1 = 2|\lambda_0|
\]
We must therefore have that $\lambda_0 = 0$ in order that the sequence $(T e_{\gamma^0_n}^*)$ has a convergent subsequence. Then, considering in turn the sequences  $(e_{\gamma^j_n}^*)_{n=1}^{\infty} \subseteq B_{\ell_1(\Gamma)}$ where $\gamma^j_n = (n+1, m_2^{-1}, 0, e_{j}^*) \in \Gamma$ (for $1 \leq j \leq k-1$), we see by the same arguments that all the $\lambda_i$ must be $0$ as claimed.
\end{proof}

It remains to see that the operator $S$ is strictly singular and that every bounded linear operator $T \in \mathcal{L}(\X_k)$ is of the form $T = \sum_{i=0}^{k-1} \lambda_i S^i + K$ for some compact operator $K \colon \X_k \to \X_k$.To begin with, we focus on proving the latter of these and aim to prove:
\begin{thm}\label{EssMainTheorem}
Let $T \colon \X_k \to \X_k$ be a bounded linear operator on $\X_k$ and $(x_i)_{i \in \N}$ a RIS in $\X_k$. Then $\dist (Tx_i, \langle x_i, Sx_i \dots ,S^{k-1}x_i \rangle_{\R}) \to 0$ as $i \to \infty$. 
\end{thm}

The proof is similar to that given in \cite{AH}. We will need slight modifications to the definitions of exact pairs and dependent sequences. We find it convenient to define both the $0$ ($\delta = 0$ in the definitions that follow) and $1$ ($\delta = 1$ in the definitions that follow) exact pairs and dependent sequences below. However, initially, we will only be concerned with the $0$ exact pairs and dependent sequences. The $1$ exact pairs and dependent sequences will only be needed to establish that the space $\X_k$ is hereditarily indecomposable. We also introduce the new, but related notions of `weak exact pairs' and `weak dependent sequences' which will be useful to us later in establishing strict singularity of $S$. 

\begin{defn} \label{SpecialExactPair}

Let $C>0, \, \delta \in \{0 ,1\}$. A pair $(x,\eta)\in
\X_k\times \Gamma$ is said to be a $(C,j,\delta )$-{\em special exact pair}
if
 \begin{enumerate}
 \item $\|x\| \leq C$
 \item $|\langle d^*_\xi,x\rangle|\le Cm_{j}^{-1}$ for all $\xi\in \Gamma$;
 \item $\weight \eta = m_{j}^{-1}$
 \item $x(\eta) = \delta$ and $S^l x(\eta) = 0$ for every $1 \leq l \leq k-1 $
 \item for every element $\eta'$ of
$\Gamma$ with $\weight\eta'= m_{i}^{-1}\ne m_j^{-1}$, we have
$$
|x(\eta')| \le \begin{cases} Cm_i^{-1} &\text{ if } i<j\\
                 Cm_{j}^{-1} &\text{ if }i>j.
                            \end{cases}
$$
\end{enumerate}
Given also an $\ve > 0$ we will say a pair $(x,\eta)\in
\X_k\times \Gamma$ is a $(C,j, 0, \ve)$-{\em weak exact pair}
if condition (4) is replaced by the following (weaker) condition
\begin{enumerate}
\item[4$'$] $ |S^l x ( \eta) | \leq C \ve$ for $ 0 \leq l \leq k-1$
\end{enumerate}
We will say a pair $(x, \eta) \in \X_k \times \Gamma$ is a $(C, j, 1, \ve)$-{\em weak exact pair} if condition (4) is replaced by condition 
\begin{enumerate}
\item[4$''$] $x(\eta) = 1$ and $|S^lx(\eta) | \leq C\ve$ for $1 \leq l \leq k-1$
\end{enumerate}
\end{defn}
We note that a $(C, j, \delta)$ special exact pair is a $(C, j, \delta, \ve)$ weak exact pair for any $\ve > 0$. Moreover, the definition of a $(C,j,\delta)$-special exact pair is the same as the definition of a $(C, j, \delta)$ exact pair given in \cite{AH} but with the additional requirement that $S^jx(\eta) = 0$ for all $j$. The remark made in AH (following the definition of exact pairs) is therefore still valid. In fact it is easily verified that the same remark in fact holds for weak exact pairs.  For convenience, we state the remark again as it will be useful to us later:
\begin{remark}
A $(C,j, \delta, \ve)$ weak exact pair also satisfies the estimates
$$
|\langle e^*_{\eta'},P_{(s,\infty)}x\rangle| \le \begin{cases} 6Cm_i^{-1} &\text{ if } i<j\\
                 6Cm_{j}^{-1} &\text{ if }i>j
                            \end{cases}
$$
for elements $\eta'$ of $\Gamma$ with $\weight \eta'=m_i^{-1}\ne
m_j^{-1}$.
\end{remark}

We will need the following method for constructing 0 special exact pairs.
\begin{lem}\label{RISZeroSpecialExact}
Let $(x_k)_{k=1}^{n_{2j}}$ be a skipped-block $C$-RIS, and let
$q_0<q_1<q_2<\cdots<q_{n_{2j}}$ be  natural numbers such that $\ran
x_k\subseteq (q_{k-1},q_k)$ for all $k$. Let $z$ denote the weighted
sum $z=m_{2j}n_{2j}^{-1}\sum_{k=1}^{n_{2j}}x_k$. For each $k$ let
$b^*_k$ be an element of $B_{q_{k-1},q_{k}-1}$ with $\langle
b^*_k,x_k\rangle=0$ and $\langle
(S^*)^l b^*_k,x_k\rangle= \langle b^*_k , S^l x_k \rangle = 0$ for all $l$.Then there exist $\zeta_i \in \Delta_{q_i}$
($1\le i\le n_{{2j}}$) such that the element $\eta=\zeta_{n_{2j}}$
has analysis $(q_i,b^*_i,\zeta_i)_{1\le i\le n_{2j}}$ and $
(z,\eta)$ is a $(16C,{2j},0)$-special exact pair.
\end{lem}

\begin{proof}
The proof is the same as in AH. We only need to show that $S^lz(\eta) = 0$ for $1\leq l \leq k-1$. This is easy. 
\[
S^lz(\eta) = \langle S^lz, e_{\eta}^* \rangle = \langle S^lz, \sum_{k=1}^{n_{2j}} d_{\zeta_k}^* + m_{2j}^{-1} P_{(q_{k-1},q_k)}^* b_k^* \rangle
\]
It is clear from the definition of $S$ that $\ran x_k \subseteq (q_{k-1},q_k) \implies \ran S^lx_k \subseteq (q_{k-1},q_k)$ and since $\rank\zeta_k = q_k$ for every $k$, it follows that $\langle S^lz, \sum_{k=1}^{n_{2j}} d_{\zeta_k}^* \rangle = 0$. We thus see that \[
S^lz(\eta) = n_{2j}^{-1} \sum_{k=1}^{n_{2j}} \langle S^lx_k , b_k^* \rangle = 0
\]
as required.
\end{proof}

\begin{defn}\label{DepSeq}Consider the space $\X_k$. We
shall say that a sequence $(x_i)_{i\le n_{2j_0-1}}$ is a $(C,
2j_0-1,\delta)$-{\em special dependent sequence} if there exist
$0=p_0<p_1<p_2<\cdots<p_{n_{2j_0-1}}$, together with
$\eta_i\in\Gamma_{p_i-1}\setminus \Gamma_{p_{i-1}}$ and $\xi_i\in
\Delta_{p_i}$ ($1\le i\le n_{2j_0-1}$) such that
\begin{enumerate}
\item for each $k$, $\ran x_k\subseteq (p_{k-1},p_k)$;
 \item the
element $\xi=\xi_{n_{2j_0-1}}$ of $\Delta_{p_{n_{2j_0-1}}}$ has weight
$m^{-1}_{2j_0-1}$ and analysis
$(p_i,e^*_{\eta_i},\xi_i)_{i=1}^{n_{2j_0-1}}$;
\item $(x_1,\eta_1)$ is a $(C,4j_1,\delta)$-special exact pair; \item
for each $2\le i\le n_{2j_0-1}$, $(x_i,\eta_i)$ is a
$(C,4\sigma(\xi_{i-1}),\delta)$-special exact pair.
\end{enumerate}
If we instead ask that 
\begin{enumerate}
\item[3$'$] $(x_1, \eta_1)$ is a $(C, 4j_1, \delta, n_{2j_0-1}^{-1})$ weak exact pair
\item[4$'$]  $(x_i , \eta_i)$ is a $(C, 4\sigma(\xi_{i-1}), \delta, n_{2j_0-1}^{-1})$ weak exact pair for $2 \leq i \leq n_{2j_0-1}$
\end{enumerate}
we shall say the sequence $(x_i)_{i=0}^{n_{2j_0-1}}$ is a {\em weak $(C, 2j_0-1, \delta)$ dependent sequence}.

In either case, we notice that, because of the special odd-weight conditions, we necessarily have $m^{-1}_{4j_1} = \weight \eta_1
<n^{-2}_{2j_0-1}$, and $\weight \eta_{i+1} =m^{-1}_{4\sigma(\xi_i)} < n^{-2}_{2j_0-1}$,
by lemma \ref{monotonicity_of_odd_weights} for $1\le i<n_{2j_0-1}$. 
\end{defn}
We also observe that a $(C, 2j_0 - 1, 0)$ special dependent sequence is certainly a weak $(C, 2j_0 - 1, 0)$ dependent sequence.
\begin{lem}\label{ZeroDepSeqLem}
Let $(x_i)_{i\le n_{2j_0-1}}$ be a weak $(C, 2j_0-1,0)$-{dependent
sequence} in $\X_k$ and let $J$ be a sub-interval of
$[1,n_{2j_0-1}]$. For any $\gamma'\in \Gamma$ of weight $m_{2j_0-1}$
we have
$$
|\sum_{i\in J} x_i(\gamma')| \le 7C.
$$
\end{lem}
\begin{proof}
Let $\xi_i,\eta_i,p_i,j_1$ be as in the definition of a dependent
sequence and let $\gamma$ denote $\xi_{n_{2j_0-1}}$, an element of
weight $m_{2j_0-1}$. Let $\big( p_0', (p_i',b'^*_i,\xi'_i)_{1\le i\le
a'} \big)$ be the analysis  of $\gamma'$ where each $b'^*_i$ is either $0$ or $e_{\eta'_i}^*$ for some suitable $\eta'$. 

The proof is easy if all the $b'^*_r$ are $0$, or if \[
\{ \weight\eta'_r : 1\leq r \leq a', b'^*_r = e_{\eta'_r}^* \} \cap \{ \weight\eta_i : 1\leq i \leq a \} = \varnothing
\]

So we may suppose that there is some $1 \leq r \leq a'$ s.t. $b'^*_r = e_{\eta'_r}^*$ with $\weight\eta'_r = \weight\eta_i$ for some $i$. We choose $l$ maximal such that there exists $i$ with $\weight\eta_i'= \weight\eta_l$. Clearly we can estimate as follows\[
|\sum_{k\in J} x_k(\gamma')| \leq |\sum_{k\in J,\ k<l} x_k(\gamma')|
+ |x_l(\gamma')| +
\sum_{k\in J,\ k>l}|x_k(\gamma')|
\]
We now estimate the three terms on the right hand side of the inequality separately. $|x_l(\gamma')| \leq \|x\| \leq C$. Also

\begin{align*}
\sum_{k\in J,\ k>l}|x_k(\gamma')| &= \sum_{k\in J,\ k>l}\sum_{i\le a'}|\langle d^*_{\xi'_i},x_k\rangle + m_{2j_0-1}^{-1}\langle b'^*_i, P_{(p'_{i-1},\infty)}x_k\rangle| \\
&\leq n_{2j_0-1}^2\max_{l<k\in J,\ i\le a'}|\langle d^*_{\xi'_i},x_k\rangle + m_{2j_0-1}^{-1} \langle b'^*_i,
P_{(p'_{i-1},\infty)}x_k\rangle|
\end{align*}

Now each $b'^*_i$ is $0$ or $e_{\eta'_i}^*$ where $\weight\eta'_i \neq \weight \eta_k$ for any $k > l$ (or else we would contradict maximality of $l$). If $b'^*_i = 0$, then \[
|\langle d^*_{\xi'_i},x_k\rangle + m_{2j_0-1}^{-1} \langle b'^*_i, P_{(p'_{i-1},\infty)}x_k\rangle| = |\langle d^*_{\xi'_i},x_k\rangle | \leq C\weight\eta_k \leq C n_{2j_0 - 1}^{-2}
\]
where the penultimate inequality follows from the definition of a (weak) exact pair, and the final inequality follows from lemma \ref{monotonicity_of_odd_weights}. Otherwise $b'^*_i = e_{\eta'_i}^*$ where in particular (by restrictions on elements of odd weight) $\weight\eta'_i < n_{2j_0 - 1}^{-2}$. By the definition of (weak) exact pair and the remark following it we have \[
|\langle d^*_{\xi'_i},x_k\rangle +m_{2j_0-1}^{-1}P_{(p_{i-1},\infty)}x_k(\eta'_i)| \le
 C\weight \eta_k + 6Cm_{2j_0-1}^{-1}\max\{\weight \eta'_{i}, \weight \eta_{k}\} \leq 3Cn_{2j_0 -1}^{-2}
\]
Finally we consider $|\sum_{k\in J,\ k<l} x_k(\gamma')|$. Obviously if $l=1$ this sum is zero, and the lemma is proved. So we can suppose $l>1$. By definition of $l$, there exists some $i$ such that $b'^*_i = e_{\eta'_i}^*$ and $\weight\eta_l = \weight\eta'_i$. Now either $i = 1$ or $i > 1$. We consider the 2 cases separately.

Suppose first that $i=1$.
\begin{align*}
|\sum_{k \in J,\ k<l}x_k(\gamma')| &= |\langle \sum_{k\in J,\ k<l}x_k , \sum_{r=1}^{a'} d^*_{\xi'_r} + m_{2j_0-1}^{-1}  P^*_{(p'_{r-1},\infty)}b'^*_r \rangle | \\
&\leq n_{2j_0 - 1}^{2} \max_{J \ni k<l,\ r \leq a'} |\langle x_k,d^*_{\xi'_r} \rangle| + m_{2j_0-1}^{-1} | \langle \sum_{k\in J,\ k<l}x_k , \sum_{\substack{r\leq a' \text{s.t.} \\ b'^*_r = e_{\eta'_r}^*}} P^*_{(p'_{r-1},\infty)}e_{\eta'_r}^* \rangle | \\
&\leq C + m_{2j_0-1}^{-1} | \langle \sum_{k\in J,\ k<l}x_k , \sum_{\substack{r\leq a' \text{s.t.} \\ b'^*_r = e_{\eta'_r}^*}} P^*_{(p'_{r-1},\infty)}e_{\eta'_r}^* \rangle |
\end{align*}
where the last inequality follows once again from the definition of exact pair. Suppose that for some $k \in J$, $k < l$, there is an $r$ in $\{1, 2, \dots a' \}$ with $b'^*_r = e_{\eta'_r}^*$ and $\weight\eta'_r = \weight\eta_k$. By lemma \ref{monotonicity_of_odd_weights}, we get $\weight\eta'_r = \weight\eta_k > \weight\eta_l = \weight\eta'_1$, i.e. $\weight\eta'_r > \weight\eta'_1$. But since $\gamma'$ also has odd weight, this clearly contradicts lemma \ref{monotonicity_of_odd_weights} applied to $\gamma'$. Thus there does not exist $r$ in $\{1, 2, \dots a' \}$ with $b'^*_r = e_{\eta'_r}^*$ and $\weight\eta'_r = \weight\eta_k$ for some $k \in J$, $k<l$. Using an argument similar to the above, we finally deduce that 
\[
m_{2j_0-1}^{-1} | \langle \sum_{k\in J,\ k<l}x_k , \sum_{\substack{r\leq a' \text{s.t.} \\ b'^*_r = e_{\eta'_r}^*}} P^*_{(p'_{r-1},\infty)}e_{\eta'_r}^* \rangle | \leq 2C \]
and so we get the required result.

Finally it remains to consider what happens when $i > 1$. Recall we are also assuming $l > 1$ and $\weight\eta'_i = \weight\eta_l$. But by definition of a special exact pair, we have $\weight\eta_l = m_{4\sigma(\xi_{l-1})}^{-1}$, and by restrictions on elements of odd weights, $\weight\eta'_i = m_{4\omega}^{-1}$ with $\omega \in \Sigma(\xi'_{i-1})$. By strict monotonicity of the sequence $m_j$, we deduce that $\omega = \sigma(\xi_{l-1}) \in \Sigma(\xi'_{i-1})$. By lemma \ref{IntersectionPropertyofSigmaSets} there are now two possibilites. Either $\xi_{l-1} = \xi'_{i-1}$ or, if not, there is some $j$, $1\leq j \leq k-1$, such that $F^j(\xi_{l-1}) = \xi'_{i-1}$. In either of these cases, we note that in particular this implies $p_{l-1} = p'_{i-1}$ since $F$ preserves rank and we can write the evaluation analysis of $\gamma'$ as \[
e_{\gamma'}^* = (S^*)^j(e^*_{\xi_{l-1}}) + \sum_{r=i}^{a'} d_{\xi'_r}^* + m_{2j_0-1}^{-1} P_{(p'_{r-1},p'_r)}^*b'^*_r
\]
for some $0 \leq j \leq k-1$. Now, for $k<l$, since $\ran x_k \subseteq (p_{k-1}, p_k) \subseteq (0,p_{l-1}) = (0, p'_{i-1})$, we see that 
\begin{align*}
|\langle x_k, e_{\gamma'}^* \rangle| &= |\langle x_k, (S^*)^je_{\xi_{l-1}}^* \rangle| \\
&= |\langle S^jx_k, \sum_{r=1}^{l-1} d_{\xi_r}^* + m_{2j_0-1}^{-1}P^*_{(p_{r-1}, p_r)} e_{\eta_r}^* \rangle| \\
&= m^{-1}_{2j_0-1} |\langle S^jx_k, e_{\eta_k}^* \rangle| \\
&\leq Cn_{2j_0-1}^{-1} \text{ by definition of a weak exact pair}
\end{align*}
and so \[
|\sum_{k \in J,\ k<l}x_k(\gamma ')| \leq n_{2j_0-1} \max_{k \in J,\  k<l} |x_k(\gamma')| \leq C
\]
This completes the proof.
\end{proof}

As a consequence of the above lemma and proposition \ref{AHProp5.6} we obtain the following upper norm estimate for the averages of weak special dependent sequences.

\begin{prop} \label{0DepSeqUpperEst}
Let $(x_i)_{i \leq n_{2j_0-1}}$ be a weak $(C, 2j_0-1, 0)$ dependent sequence in $\X$. Then \[
\| n_{2j_0-1}^{-1} \sum_{i=1}^{n_{2j_0-1}} x_i \| \leq 70C m_{2j_0-1}^{-2}
\]
\end{prop}
\begin{proof}
We apply the second part of proposition \ref{AHProp5.6}, with $\lambda_i = 1$ and $2j_0-1$ playing the role of $j_0$. Lemma \ref{ZeroDepSeqLem} shows that the extra hypothesis of the second part of proposition \ref{AHProp5.6} is satisfied, provided we replace $C$ by $7C$. We deduce that $\| n_{2j_0-1}^{-1} \sum_{i=1}^{n_{2j_0-1}} x_i \| \leq 70C m_{2j_0-1}^{-2}$ as claimed.
\end{proof}

The proof of theorem \ref{EssMainTheorem} is now easy. We obtain the following, minor variation, of lemma 7.2 of \cite{AH}:
\begin{lem}\label{DistExact}
Let $T$ be a bounded linear operator on $\X$, let $(x_i)$ be a
$C$-RIS in $\X \cap \Q^\Gamma$ and assume that
$\dist (Tx_i, \langle x_i, Sx_i \dots ,S^{k-1}x_i \rangle_{\R}) >\delta>0$ for all $i$.  Then, for
all $j,p\in \N$, there exist $z\in [x_i:i\in \N]$, $q>p$ and
$\eta\in \Delta_{q}$ such that
 \begin{enumerate}
 \item $(z,\eta)$ is a $(16C,2j,0)$-special exact pair;
 \item $(Tz)(\eta)>\frac7{16}\delta$;
 \item $\|(I-P_{(p,q)})Tz\|<m_{2j}^{-1}\delta$;
 \item $\langle P^*_{(p,q]}e^*_{\eta},Tz\rangle >\frac{3}{8}\delta$.
 \end{enumerate}
\end{lem}
The proof of the theorem is now the same as the proof of \cite{AH} proposition 7.3.
\section{Operators on the Space $\X_k$}
In this section, we see that all the operators on the space $\X_k$ are expressible as $\sum_{j=0}^{k-1} \lambda_j S^j + K$ for suitable scalars $\lambda_j$ and some compact operator $K$ on $\X_k$. Before proving our main result, we prove some easy lemmas which will be of use.
\begin{lem}\label{PerturbedRIS}
Let $1\leq j \leq k-1$ and $(x_i)_{i\in\N}$ be a C-RIS in $\X_k$. Suppose there are $\widetilde{x_i}'$ such that $\| \widetilde{x_i}' - x_i \| \to 0$ as $i \to \infty$ and $S^j \widetilde{x_i}' = 0$ for every $i$. Then there is a subsequence $(x_{i_k})_{k\in\N}$ of $(x_i)$ and vectors $x_k'$ satisfying
\begin{enumerate}[(1)]
\item $\ran x_{i_k} = \ran x_k'$
\item $S^jx_k' = 0$ for every $k$
\item $\|x_k' - x_{i_k} \| \to 0$
\item $(x_k')_{k\in\N}$ is a $2C$-RIS
\end{enumerate}
We note that in particular, if $(x_i)_{i\in\N}$ is a $C$-RIS with $S^jx_i \to 0$, then the above hypothesis are satisfied as a consequence of corollary \ref{corSisQuotientOp}.
\end{lem}

\begin{proof}
Let $\ran x_i = (p_i, q_i)$ and set $y_i = P_{(p_i, q_i)} \widetilde{x_i}'$. Certainly then $\ran y_i = \ran x_i$ for every $i$. Note that $\left( I - P_{(p_i, q_i)} \right) x_i = 0$ and consequently \[
\left\| \left( I - P_{(p_i, q_i)} \right) \widetilde{x_i}' \right\| = \left\| \left( I - P_{(p_i, q_i)} \right) ( \widetilde{x_i}' - x_i ) \right\| \leq \left\| I - P_{(p_i, q_i)} \right\| \left\| \widetilde{x_i}' - x_i \right\| \leq 5 \left\| \widetilde{x_i}' - x_i \right\| \to 0
\]
It follows that
\begin{align*}
\| y_i - x_i \| &= \left\| \widetilde{x_i}' - \big( ( I - P_{(p_i, q_i)} ) \widetilde{x_i}' \big) - x_i \right\| \\
&\leq \left\| \widetilde{x_i}' - x_i \right\| + \left\| \left( I - P_{(p_i, q_i)} \right) \widetilde{x_i}' \right\| \to 0
\end{align*}
Note also that $S^jy_i = 0$ for every $i$. Indeed, for $\gamma \in \Gamma$ 

\[
S^jy_i (\gamma) = \langle S^jy_i , e_{\gamma}^* \rangle = \langle S^jP_{(p_i, q_i)} \widetilde{x_i}' , e_{\gamma}^* \rangle  
= \langle \widetilde{x_i}' , P_{(p_i, q_i)}^* (S^*)^j e_{\gamma}^* \rangle =  \langle \widetilde{x_i}' ,  (S^*)^j P_{(p_i, q_i)}^* e_{\gamma}^* \rangle = \langle S^j\widetilde{x_i}' , P_{(p_i, q_i)}^* e_{\gamma}^* \rangle = 0
\]
since $S^j\widetilde{x_i}' = 0$ for every $i$.
So far we have managed to achieve (1) - (3) of the above. We show we can extract a subsequence of the $y_i$, $(y_{i_k})_{k\in\N}$ say, such that $(y_{i_k})_{k\in\N}$ is a $2C$-RIS. The proof will then be complete if we set $x_k' = y_{i_k}$ and take the subsequence $(x_{i_k})$ of the $x_i$. 

Since $\| x_i \| \leq C $ for every $i$ and $\| y_i - x_i \| \to 0$, we can certainly assume (by ignoring some finite number of terms at the beginning of the sequence) that $\| y_i \| \leq 2C $ for every $i$. Let $(j_k)$ be the increasing sequence corresponding to the $C$-RIS $(x_i)$, i.e.
\begin{enumerate}
\item $j_{k+1} > \text{max } \ran x_k$
\item $|x_k(\gamma)| \leq Cm_{i}^{-1} \text{ when $\weight\gamma = m_{i}^{-1}$ and $i < j_k$}$
\end{enumerate}
Set $l_1 = j_1$. We can certainly find an $i_1 \geq 1$ such that $\| y_{i_1} - x_{i_1} \| \leq Cm_{l_1}^{-1}$. So if $\gamma \in \Gamma$, $\weight \gamma = m_{w}^{-1}$ with $w < l_1$, then certainly $w < l_1 = j_1 \leq j_{i_1}$ so \[
|y_{i_1} (\gamma) | \leq |(y_{i_1} - x_{i_1})(\gamma) | + |x_{i_1}(\gamma)| \leq Cm_{l_1}^{-1} + Cm_{w}^{-1} \leq 2Cm_{w}^{-1}
\]
Now set $l_2 = j_{i_1+1}$. So $l_2 > \text{max } \ran x_{i_1} = \text{max } \ran y_{i_1}$. 

Inductively, suppose we have defined natural numbers $l_1 \leq l_2 \leq \dots \leq l_n$ and $i_1 < i_2 < \dots < i_n$ such that 
\begin{enumerate}[(i)]
\item $l_{k+1} > \text{max } \ran y_{i_k}$ for all $1\leq k \leq n-1$, and
\item for all $1\leq k \leq n$, $|y_{i_k} (\gamma) | \leq 2Cm_{w}^{-1}$ whenever $\gamma \in \Gamma$, $\weight \gamma = m_{w}^{-1}$ with $w < l_k$
\end{enumerate}
Set $l_{n+1} = j_{i_n + 1}$. It is easily seen from the inductive construction that $l_{n+1} \geq l_n$ and moreover (by choice of $j_k$), $l_{n+1} > \text{max } \ran x_{i_n} = \text{max } \ran y_{i_n}$. Now we can certainly find $i_{n+1} > i_n$ such that $\|y_{i_{n+1}} - x_{i_{n+1}} \| \leq Cm_{l_{n+1}}^{-1}$. So suppose $\gamma \in \Gamma$, $\weight \gamma = m_{w}^{-1}$ with $w < l_{n+1}$ In particular $w < j_{i_n + 1} \leq j_{i_{n+1}}$ so by choice of $i_{n+1}$ and the fact that $(x_i)$ is a RIS we see that   \[
|y_{i_{n+1}} (\gamma) | \leq |(y_{i_{n+1}} - x_{i_{n+1}})(\gamma)| + |x_{i_{n+1}}(\gamma)| \leq Cm_{l_{n+1}}^{-1} + Cm_{w}^{-1} \leq 2Cm_{w}^{-1}
\]

Inductively we obtain a subsequence $(y_{i_k})_{k\in\N}$ which is evidently a $2C$-RIS (with the sequence $(l_k)_{k\in\N}$ satisfying the RIS definition), as required.
\end{proof}
\begin{lem} \label{TechnicalLemma}
Suppose $(x_i)_{i\in \N}$ is a normalised sequence in $\X_k$ and $\lambda_j \in \R$ ($0\leq j \leq k-1$) are scalars such that $\sum_{j=0}^{k-1} \lambda_j S^j x_i \to 0$. If $S^{k-1} x_i \arrownot\to 0$ then $\lambda_j = 0$ for every $j$. Otherwise, there is $1 \leq m \leq k-1$ such that $S^m x_i \to 0$ but $S^j x_i \arrownot\to 0$ if $j < m$, in which case, we must have that $\lambda_j = 0$ for all $j < m$.
\end{lem}

\begin{proof}
We consider first the case where $S^{k-1} \to 0$ and choose $m \in \{ 1, \dots k-1 \}$ minimal such that $S^mx_i \to 0$ (noting such an $m$ obviously exists). We must observe that $\lambda_j =  0$ for all $j<m$. Since $S^j x_i \to 0$ for all $j \geq m$ we in fact know that $\sum_{j=0}^{m-1} \lambda_j S^j x_i \to 0$. If $m = 1$, this of course implies that $\lambda_0 =  0$ since the sequence $(x_i)$ is normalised and we are done. 

Otherwise, we apply the operator $S^{m-1}$ to the previous limit and deduce that $\lambda_0 S^{m-1} x_i \to 0$ (again making use of the fact that $S^jx_i \to 0$ when $j \geq m$). Since, by choice of $m$, $S^{m-1} x_i \arrownot\to 0$, we must again have $\lambda_0 = 0$, and moreover, $\sum_{j=1}^{m-1} \lambda_j S^j x_i \to 0$. If $m = 2$, then this implies $\lambda_1 Sx_i \to 0$ which implies that $\lambda_1 = 0$ (since $Sx_i \arrownot\to 0$). Otherwise, we apply the operator $S^{m-2}$. A similar argument concludes once again that we must have $\lambda_1 = 0$. Continuing in this way, we get that $\lambda_j = 0$ for all $j < m$ as required.

In the case where $S^{k-1}x_i \arrownot\to 0$, we notice that in particular this implies $S^jx_i \arrownot\to 0$ for every $1 \leq j \leq k-1$. Applying the operators $S^{k-1}, S^{k-2}, \dots S$ sequentially to the limit $\sum_{j=0}^{k-1} \lambda_j  S^jx_i \to 0$ yields first that $\lambda_0 = 0$, then $\lambda_1 = 0$ etc. So $\lambda_j = 0$ for every $j$ as required.
\end{proof}

\begin{thm}
Let $T \colon \X_k \to \X_k$ be a bounded linear operator on $\X_k$. Then there are $\lambda_j \in \R$ ($0 \leq j \leq k-1$) and a compact operator $K: \X_k \to \X_k$ such that $T = \sum_{j=0}^{k-1} \lambda_j S^j + K$.
\end{thm}

\begin{proof}
We will show that there exist $\lambda_j$ such that whenever $(x_i)_{i \in \N}$ is a RIS, $Tx_i - \sum_{j=0}^{k-1} \lambda_j S^j x_i \to 0$ as $i \to \infty$. By proposition \ref{RIStoBlock}, this implies $Tx_i - \sum_{j=0}^{k-1} \lambda_j S^j x_i \to 0$ for every block sequence $(x_i)$ which, of course, implies that $T-\sum_{j=0}^{k-1} \lambda_j S^j$ is compact. We note that it is enough to show that there are $\lambda_j \in \R$ such that whenever $(x_i)_{i\in\N}$ is a RIS, we can find some subsequence $(x_{i_l})$ with $Tx_{i_l} - \sum_{j=0}^{k-1}\lambda_j S^j x_{i_l} \to 0$. 

\begin{nclaim} \label{RISSxiNotTo0}
Suppose $(x_i)_{i\in\N}$ is a normalised RIS and that $S^{k-1} x_i \arrownot\to 0$ (noting in particular that this implies that $S^jx_i \arrownot\to 0$ for every $1 \leq j \leq k-1$). Then there are $\lambda_j \in \R$ ($0 \leq j \leq k-1$) and a subsequence $(x_{i_l})$ of $(x_i)$ such that $Tx_{i_l} - \sum_{j=0}^{k-1} \lambda_j S^j x_{i_l} \to 0$. 
\end{nclaim}
By passing to a subsequence, we may assume that there is some $\ve > 0$ such that $\| S^jx_i \| \geq \ve$ for every $i\in \N$ and every $1\leq j \leq k-1$. By theorem \ref{EssMainTheorem}, there are $\lambda^j_i  \in \R$ such that $\| Tx_i - \sum_{j=0}^{k-1} \lambda^j_i S^j x_i \| \to 0$. We first show that the $\lambda^0_i$ must converge. The argument is similar to that of \cite{AH}. If not, by passing to a subsequence, we may assume that $|\lambda^0_{i+1} - \lambda^0_{i}| \ge \delta > 0$ for some $\delta$. Since $y_i := x_{2i-1} + x_{2i}$ is a RIS, we deduce from theorem \ref{EssMainTheorem} that there are $\mu^j_i  \in \R$ with $\|Ty_i - \sum_{j=0}^{k-1} \mu^j_i S^j y_i \| \to 0$. Now
\begin{align*}
\| \sum_{j=0}^{k-1} \left( \lambda^j_{2i} - \mu^j_i \right) S^j x_{2i} &+ \sum_{j=0}^{k-1} \left( \lambda_{2i-1}^j - \mu^j_i \right) S^j x_{2i-1} \|  \\
& \leq \| \sum_{j=0}^{k-1} \lambda_{2i}^{j} S^jx_{2i} - Tx_{2i} \|+ \sum_{j=0}^{k-1} \lambda_{2i-1}^{j} S^jx_{2i-1} - Tx_{2i-1} \| + \| Ty_i - \sum_{j=0}^{k-1} \mu_i^j S^j y_i \| 
\end{align*}
and so we deduce that both sides of the inequality converge to $0$. Since the sequence $(x_i)$ is a block sequence, there exist $l_k$ such
that $P_{(0,l_k]}y_k = x_{2k-1}$ and $P_{(l_k,\infty)}y_k = x_{2k}$ Recalling that if $x \in \X_k$ has $\ran x = (p,q] $ then $\ran S^jx = (p,q]$, we consequently have \[
\| \sum_{j=0}^{k-1} \left( \lambda^j_{2i} - \mu^j_i \right) S^j x_{2i} \| \leq  \| P_{(l_k, \infty)} \| \| \sum_{j=0}^{k-1} \left( \lambda^j_{2i} - \mu^j_i \right) S^j x_{2i} + \sum_{j=0}^{k-1} \left( \lambda_{2i-1}^j - \mu_i^j \right) S^j x_{2i-1} \|  \to 0
\]
and similarly
\[
\| \sum_{j=0}^{k-1} \left( \lambda^j_{2i-1} - \mu^j_i \right) S^j x_{2i-1} \| \to 0
\]
By continuity of $S$ and the fact that $S^k = 0$, applying $S^{k-1}$ to both limits above (and recalling that $\| S^j x_i \| \geq \varepsilon$ for every $1 \leq j \leq k-1$) we obtain 
\[
|\lambda^0_{2i} - \mu^0_{i} | \leq \frac{1}{\varepsilon} \| \left( \lambda^0_{2i} - \mu^0_{i} \right) S^{k-1} x_{2i} \| \leq \frac{1}{\varepsilon} \| S^{k-1} \| \| \sum_{j=0}^{k-1} \left( \lambda_{2i}^j - \mu^j_i \right) S^j x_{2i} \| \to 0
\]
and similarly we find that $|\lambda^0_{2i-1} - \mu^0_{i} | \to 0$. It follows that $| \lambda^0_{2i} - \lambda^0_{2i-1} | \to 0$ contrary to our assumption. So the $\lambda^0_i$ converge to some $\lambda_0$ as claimed. It follows that $\| Tx_i - \lambda_0 x_i - \sum_{j=1}^{k-1} \lambda^j_i  S^jx_i \| \to 0$. We observe that, since $(x_i)$ is normalised and $\| S^{k-1} x_i \| \geq \varepsilon$, applying $S^{k-2}$ to the previous limit, we see that \[
|\lambda^1_i | \leq \frac{1}{\varepsilon} \| \lambda^1_i S^{k-1}x_i \| \leq \| S^{k-2} \| \| Tx_i - \lambda_0 x_i - \sum_{j=1}^{k-1} \lambda^j_i S^jx_i \| + \| S^{k-2}T - \lambda_0 S^{k-2} \| < \infty
\]
so that in particular the $\lambda_i^1$ are bounded. Consequently there is some convergent subsequence $\lambda_{i_l}^1$ (limit $\lambda_1$) of the $\lambda^1_i$. It follows that the corresponding subsequence $(x_{i_l})$ satisfies \[
Tx_{i_l} - \lambda_0 x_{i_l} - \lambda_1 Sx_{i_l} - \sum_{j=2}^{k-1}\lambda^j_{i_l} S^j x_{i_l} \to 0
\]
Now, if $k=2$ we are done (the last sum is empty). Otherwise, we can apply $S^{k-3}$ to the previous limit and use the same argument to conclude that $(\lambda^2_{i_l})_{l=1}^{\infty}$ is a bounded sequence of scalars. Continuing in this way, we eventually find (after passing to further subsequences which we relabel as $x_{i_l}$) that there are $\lambda_j$ with $(T - \sum_{j=0}^{k-1} \lambda_j S^j )x_{i_l} \to 0$ as required.
\begin{nclaim} \label{RISwithSxiTo0}
Suppose $(x_i)_{i \in \N}$ is a normalised $C$-RIS and that $S^m x_i$ converges to $0$ for some $1 \leq m \leq k-1$. Let $m_0 \geq 1$ be minimal such that $S^{m_0} x_i \to 0$. Then there are $\lambda_j \in \R$ ($0 \leq j < m_0$) and a subsequence $(x_{i_l})$ of $(x_i)$ such that $(T-\sum_{j=0}^{m_0 - 1} \lambda_j S^j) x_{i_l} \to 0$. 
\end{nclaim}
By minimality of $m_0$, we can assume (by passing to a subsequence if necessary) that $\| S^j x_i \| \geq \ve$ for all $i\in \N$ and all $j < m_0$.  By lemma \ref{PerturbedRIS}, (with $j = m_0$) there is a $2C$-RIS $(x'_l)_{l \in \N}$ and some subsequence $(x_{i_l})$ of $(x_i)$ such that $x'_l \in \Ker S^{m_0} \subseteq \Ker S^j$ for $j \ge m_0$ and every $l$. Moreover,  $\|x_{i_l} - x'_l \| \to 0$ (as $l \to \infty$). By theorem \ref{EssMainTheorem}, there are $\lambda^j_l$ s.t \[
\| Tx'_l - \sum_{j=0}^{k-1}\lambda^j_l S^j x'_l  \| = \| Tx'_l - \sum_{j=0}^{m_0-1} \lambda^j_l S^j x'_l \| \to 0
\]
We claim the $\lambda^0_l$ must converge to some $\lambda_0$. The argument is the same as that used in claim \ref{RISSxiNotTo0}, except now we obtain \[
\| \sum_{j=0}^{m_0-1} \left( \lambda^j_{2l-1} - \mu^j_l \right) S^j x'_{2l-1} \| \to 0 \text{ and } \| \sum_{j=0}^{m_0-1} \left( \lambda^j_{2l} - \mu^j_l \right) S^j x'_{2l} \| \to 0
\] 
(We note there are no terms of the form `$S^jx$' when $j \geq m_0$ since the RIS $(y_k)$ defined by $y_k := x_{2k-1}' + x_{2k}'$ also lies in $\Ker S^{m_0}$). We apply $S^{m_0-1}$, noting that $S^j x_l' = 0$ for every $l$ and $j \geq m_0$ and it follows as before that $| \lambda^0_{2j} - \lambda_{2j-1}^0 | \to 0$.
It easily follows that $Tx'_k - \lambda_0 x'_l - \sum_{j=1}^{m_0-1} \lambda^j_l S^j x'_l \to 0$. We use the same argument as above to show that the sequences $(\lambda^j_l)_{l=1}^{m_0-1}$ all have convergent subsequences and consequently that we can find some subsequence $x'_{l_r}$ with $(T - \sum_{j=0}^{m_0 - 1}\lambda_j S^j )x'_{l_r} \to 0$. It follows that
\begin{align*}
\| Tx_{i_{l_r}} - \sum_{j=0}^{m_0-1} \lambda_j S^j x_{i_{l_r}} \| &\leq \| \big( T - \sum_{j=0}^{m-1} \lambda_j S^j  \big)(x_{i_{l_r}} - x'_{l_r})\| + \| \big( T - \sum_{j=0}^{m_0-1} \lambda_j S^j \big)(x'_{l_r}) \| \\
&\leq \| T - \sum_{j=0}^{m_0-1}\lambda_j S^j  \| \|x_{i_{l_r}} - x'_{l_r} \| + \| \big( T - \sum_{j=0}^{m_0-1} \lambda_j S^j \big)(x'_{l_r}) \| \to 0
\end{align*}
A priori, the $\lambda_j$ found in claims \ref{RISSxiNotTo0} and \ref{RISwithSxiTo0} may depend on the RIS. We see now that this is not the case.
\begin{nclaim} \label{sameLambdaAndMuWork}
There are $\lambda_j \in \R$ ($0 \leq j \leq k-1$) such that whenever $(x_i)_{i\in\N}$ is a RIS, there is a subsequence $(x_{i_l})$ of $(x_i)$ such that $Tx_{i_l} - \sum_{j=0}^{k-1}\lambda_j S^j x_{i_l} \to 0$.
\end{nclaim}
If $(x_i)$ is a RIS with some some subsequence converging to $0$ then any $\lambda_j$ can be chosen satisfying the conclusion of the claim. So it is sufficient to only consider normalised RIS. Let $(x_i)_{i\in\N}, (x_i')_{i\in\N}$ be normalised RIS. It follows from claims \ref{RISSxiNotTo0} and \ref{RISwithSxiTo0} that, after passing to subsequences (and relabelling), there are $\lambda_j, \lambda_j'  \in \R$ with 
\begin{equation}\label{eqn1}
Tx_i - \sum_{j=0}^{k-1} \lambda_jS^j x_i \to 0 \tag{1}
\end{equation}
and
\begin{equation} \label{eqn2}
Tx_i' - \sum_{j=0}^{k-1} \lambda'_j S^j x_i' \to 0 \tag{2}
\end{equation}
To prove the claim, we must see its possible to arrange that $\lambda_j = \lambda_j'$ for every $j$. We pick natural numbers $i_1 < i_2 < \dots$ and $j_1 < j_2 \dots$ such that $\text{max } \ran x_{i_k} < \text{min } \ran x'_{j_k} \leq \text{max } \ran x'_{j_k} < \text{min } \ran x_{i_{k+1}}$ for every $k$ and such that the sequence $(x_{i_k} + x'_{j_k} )_{k\in\N}$ is again a RIS. For notational convenience, we (once again) relabel the subsequences $(x_{i_k})$, $(x_{j_k}')$ by $(x_i)$ and $(x_i')$. So (by choice of the subsequences) $(x_i + x'_i)_{i \in \N}$ is a RIS and there are natural numbers $l_i$ such that $P_{(0, l_i]} (x_i+ x'_i) = x_i$ and $P_{(l_i, \infty)} (x_i + x'_i) = x'_i$. It follows again from claims  \ref{RISSxiNotTo0} and \ref{RISwithSxiTo0} that there are $\mu_j$ and a subsequence $(x_{i_m} + x_{i_m}')$ such that 
\begin{equation} \label{eqn3}
T(x_{i_m} +x'_{i_m}) - \sum_{j=0}^{k-1} \mu_j S^j (x_{i_m} + x'_{i_m} ) \to 0 \tag{3}
\end{equation}
We note aso that 
\begin{align*}
P_{(0, l_i] }Tx_i' &= P_{(0, l_i]}\big( Tx_i' - \sum_{j=0}^{k-1} \lambda_j' S^j x_i' \big) + \sum_{j=0}^{k-1} \lambda_j' P_{(0, l_i]} S^jx_i'  \\
&= P_{(0, l_i]}\big( Tx_i' - \sum_{j=0}^{k-1} \lambda_j' S^jx_i'  \big)  \to 0
\end{align*}
and similarly, $P_{(l_i, \infty)} Tx_i \to 0$. Passing to the appropriate subsequences of equations \eqref{eqn1} and \eqref{eqn2} and substracting them from equation \eqref{eqn3} we see that 

\begin{equation} \label{eqn4}
Tx_{i_m}' - \sum_{j=0}^{k-1}(\mu_j - \lambda_j) S^j x_{i_m} - \sum_{j=0}^{k-1} \mu_j S^j x_{i_m}'  \to 0 \tag{4}
\end{equation}
and
\begin{equation} \label{eqn5}
Tx_{i_m} - \sum_{j=0}^{k-1}(\mu_j - \lambda'_j) S^j x'_{i_m} - \sum_{j=0}^{k-1} \mu_j S^j x_{i_m}  \to 0 \tag{5}
\end{equation}
Finally we apply the projections $P_{(0, l_{i_m}]}$ and $P_{(l_{i_m}, \infty)}$ to equations \eqref{eqn4} and \eqref{eqn5} respectively to obtain (using the above observations) that \[
 \sum_{j=0}^{k-1}(\mu_j - \lambda_j) S^j x_{i_m} \to 0 \text{ and } \sum_{j=0}^{k-1}(\mu_j - \lambda'_j) S^j x'_{i_m}  \to 0
\]

We now consider three cases:
\begin{enumerate}[(i)]
\item $S^{k-1}x_{i_m} \arrownot\to 0$ and $S^{k-1} x'_{i_m} \arrownot\to 0$. By lemma \ref{TechnicalLemma} and the two limits above, we see that we must have $\lambda_j = \mu_j = \lambda'_j$ for every $j$ as required.
\item There is some $1 \leq r \leq k-1$ such that $S^r x_{i_m} \to 0$, but $S^j x_{i_m} \arrownot\to 0$ for any $j < r$ and $S^{k-1}x_{i_m}' \arrownot\to 0$ (or the same but with $x_{i_l}$ and $x_{i_m}'$ interchanged). Again, by lemma \ref{TechnicalLemma}, we must have $\lambda_j' = \mu_j$ for every $j$, and $\lambda_j = \mu_j$ for all $j < r$. For $j \geq r$ we might as well assume that the $\lambda_j$ were chosen to be equal to $\lambda_j'$ since if $S^j x_{i_m} \to 0$ we can replace $\lambda_j$ by any scalar and equation (1) still holds (after passing to an appropriate subsequennce).
\item There are $1 \leq q , r \leq k-1$ with $S^q x_{i_m} \to 0$, $S^r x'_{i_m} \to 0$ but $S^j x_{i_m} \arrownot\to 0$ if $j < q$ and $S^j x'_{i_m} \arrownot\to 0$ if $j < r$. Without loss of generality we assume $q \leq r$. By the same argument as in case (ii), we can assume that $\lambda_j$ are chosen such that $\lambda_j = \lambda_j'$ whenever $j \geq q$. For $j < q \leq r$ we must have $\lambda_j = \mu_j = \lambda'_j$ by lemma \ref{TechnicalLemma}.
\end{enumerate}

In all possible cases, we have seen that we can arrange $\lambda_j = \lambda_j'$ for all $j$. This completes the proof of the claim and thus (as noted earlier), the proof. 
\end{proof}

\section{Strict Singularity of $S\colon \X_k \to \X_k$}

We see now hat $S$ is strictly singular. It is is enough to see that $S$ is not an isomorphism when restricted to any infinite dimensional block subspace $Z$ of $\X_k$. 

To establish the strict singularity of $S$, we begin by stating a result taken from the paper of Argyros and Hadyon (\cite{AH}, corollary 8.5). The reader can check that the same proofs as given in \cite{AH} will also work in the space $\X_k$ constructed here.

\begin{lem}\label{ExistRIS}
Let $Z$ be a block subspace of $\X_k$, and let $C>2$ be a real number.
Then $Z$ contains a normalized $C$-RIS.
\end{lem}

We will need a variation of lemma \ref{RISZeroSpecialExact} to be able to construct weak dependent sequences. We first observe that the lower norm estimate for skipped block sequences given in proposition 4.8 of \cite{AH} also holds in the space $\X_k$ and exactly the same proof works. We state it for here for convenience:

\begin{lem} \label{AHProp4.8}
Let $(x_r)_{r=1}^a$ be a skipped block
sequence in $\X_k$.  If $j$ is a
positive integer such that $a\le n_{2j}$ and $2j<\min\ran x_2$, then
there exists an element $\gamma$ of weight $m_{2j}^{-1}$ satisfying
\begin{align*}
\sum_{r=1}^a x_r(\gamma)&\ge \half m_{2j}^{-1} \sum_{r=1}^a
\|x_r\|.\\
\intertext{Hence} \|\sum_{r=1}^a x_r\|&\ge \half m_{2j}^{-1}
\sum_{r=1}^a \|x_r\|.
\end{align*}
\end{lem}

\begin{lem} \label{WeakExactPairs}
Let $\ve > 0$, $j$ be a positive integer and let $(x_i)_{i=1}^{n_{2j}}$ be a
skipped-block $C$-RIS, such that $\min\ran x_2 > 2j$. Suppose further that one of the following hypotheses holds:
\begin{enumerate}[(i)]
\item $\|S^{k-1} x_i\|\ge \delta$ for all $i$ (some $\delta > 0$)
\item There is some $2 \leq m \leq k-1$ (where we are, of course, assuming here that $k > 2$) such that $\| S^{m-1} x_i \| \geq \delta $ for all $i$ (some $\delta > 0$) and $\| S^m x_i \| \leq Cm_{2j}^{-1} \ve$.
\end{enumerate}
Then there exists $\eta\in \Gamma$ such that $x(\eta) \geq \frac{\delta}{2}$ where $x$ is the weighted
sum
$$
 x=m_{2j}n_{2j}^{-1}\sum_{i=1}^{n_{2j}}x_i.
$$
Moreover, if hypothesis (i) above holds, the pair $(Sx, \eta)$ is a $(16C, 2j, 0)$-special exact pair. Otherwise, hypothesis (ii) holds and $(Sx, \eta)$ is a $(16C, 2j, 0, \ve)$ weak exact pair.
\end{lem}
\begin{proof}
Let us consider first the case where hypothesis (i) holds. Since $(S^{k-1}x_i)_{i=1}^{n_{2j}}$ is a skipped block sequence, it follows from \ref{AHProp4.8} that there is an element $\gamma \in \Gamma$ of weight $m_{2j}^{-1}$ satisfying \[
m_{2j}n_{2j}^{-1} \sum_{i=1}^{n_{2j}} S^{k-1}x_i (\gamma) \ge \half n_{2j}^{-1} \sum_{i=1}^{n_{2j}} \|S^{k-1}x_i\| \ge \frac{\delta}{2}
\]
Consequently, we must have $F^{k-1}(\gamma)$ being defined, and $x(F^{k-1}(\gamma)) \ge \frac{\delta}{2}$. We set $\eta = F^{k-1}(\gamma) \in \Gamma$. Certainly $S^jSx(\eta) = \langle x, (S^*)^{j+1} e_{\eta}^* \rangle = 0$ for any $j \geq 0$(since, by lemma \ref{G^kalwaysundefined}, we must have $F(\eta) = F^{k}(\gamma)$ being undefined. So conditions 3 and 4 are satisfied for $(Sx, \eta)$ to be a $(16C, 2j, 0)$-special exact pair. The other conditions are satisfied since we know (by lemma \ref{RISZeroSpecialExact}) that they are satisfied for $x$, and the fact that for any $\theta \in \Gamma$ \[
 \langle Sx , e_{\theta}^* \rangle = \begin{cases} 0 & \text{ if $\theta \in \Gamma^1$ } \\ \langle x, e_{F(\theta)}^*\rangle & \text{ otherwise } \end{cases}
 \]
 and similarly
 \[
 \langle Sx , d_{\theta}^* \rangle = \begin{cases} 0 & \text{ if $\theta \in \Gamma^1$ } \\ \langle x, d_{F(\theta)}^*\rangle & \text{ otherwise } \end{cases}
 \]
 In the case where hypothesis (ii) holds, we find by the same argument as above that there is a $\gamma \in \Gamma$ of weight $m_{2j}^{-1}$ with $F^{m-1}(\gamma)$ being defined and $x(F^{m-1}(\gamma)) \geq \frac{\delta}{2}$. We now set $\eta = F^{m-1}(\gamma)$. Now, for any $0 \leq j \leq k-1$, either $S^j Sx (\eta) = 0$ (if $F^{j+1}(\eta) = F^{j+m}(\gamma)$ is undefined) or $|S^j Sx(\eta)| = |x(F^{j+m}(\gamma))| = |S^{j+m}x (\gamma)| \leq \| S^{j+m}x \| \leq m_{2j}n_{2j}^{-1}\sum_{i=1}^{n_2j} \| S^{j+m} x_i \| \leq C\ve$. The final inequality here is a consequence of the hypothesis when $j = 0$, and then a consequence of the fact that $S$ has norm at most $1$ for $j >0$. So certainly conditions (3) and (4$'$) hold for $(Sx, \eta)$ to be a $(16C, 2j, \ve)$ weak exact pair. The remaining conditions hold once again because they hold for $x$.
\end{proof}
\begin{thm} \label{SisStrictlySingular}
The operator $S\colon\X\to\X$ is strictly singular.
\end{thm}
\begin{proof}
We suppose by contradiction that $S$ is not strictly singular. It follows that there is some infinite dimensional block subspace $Y$ of $\X$ on which $S$ is an isomorphism, i.e. there is some $0<\delta \leq 1$ such that whenever $y \in Y, \|Sy\| \geq \delta \|y\|$. By lemma \ref{ExistRIS}, $Y$ contains a normalised skipped block $3$-RIS, $(x_i)_{i\in \N} \subseteq Y$. We note that certainly $Sx_i \arrownot\to 0$ and consider two possibilities. Either $S^{k-1} x_i \to 0$ or it does not. In the latter of these possibilities, passing to a subsequence, we can assume without loss of generality that $\|S^{k-1}x_i\| \geq \nu > 0$ for every $i$ (and some $\nu$). Thus, we see by lemma \ref{WeakExactPairs} that we can construct $(C, 2j, 0)$ special exact pairs $(Sx, \eta)$ for any $j \in \N$, with $\text{min }\ran Sx$ arbitrarily large and $x(\eta) \geq \frac{\nu}{2}$.

Otherwise, we must have $k > 2$ and there is an $m \in \{ 2, \dots, (k-1) \}$ with $S^{m-1}x_i \arrownot\to 0$ but $S^m x_i \to 0$. By passing to a subsequence, we can assume that $\| S^{m-1}x_i \| \geq \nu$ for all $i$. Moreover, for a fixed $j_0 \in \N$, since $S^mx_i \to 0$, given any $j \in \N$, we can find an $N_j \in \N$ such that $\| S^mx_i \| \leq Cm_{2j}^{-1} n_{2j_0-1}^{-1}$ for every $i \geq N_j$. So by lemma \ref{WeakExactPairs}, we can construct weak $(C, 2j,0,  n_{2j_0 - 1}^{-1})$ exact pairs $(Sx, \eta)$ for any $j \in \N$, with $\text{min }\ran Sx$ arbitrarily large and $x(\eta) \geq \frac{\nu}{2}$.

Now, we choose $j_0, j_1$ with $m_{2j_0-1} > 6720\delta^{-1}\nu^{-1}$ and $m_{4j_1} > n_{2j_0-1}^{2}$. By lemma \ref{WeakExactPairs}, and the argument above, there is a $y_1 \in Y, \eta_1 \in \Gamma$ such that $(Sy_1, \eta_1)$ is a $(48, 4j_1, 0,  n_{2j_0-1}^{-1})$-weak exact pair and $y_1(\eta_1) \ge \frac{\nu}{2}$. We let $p_1 > \rank \eta_1 \vee \text{max }\ran y_1$ and define $\xi_1 \in \Delta_{p_1}$ to be $(p_1, 0, m_{2j_0-1}, e_{\eta_1}^*)$.

Now set $j_2 = \sigma (\xi_1)$. Again by lemma \ref{WeakExactPairs} and the argument above, there is $y_2 \in Y, \eta_2 \in \Gamma$ with $\text{min }\ran y_2 > p_1, y_2(\eta_2) \ge \frac{\nu}{2}$ and $(Sy_2, \eta_2)$ a $(48, 4j_2, 0, n_{2j_0-1}^{-1})$-weak exact pair. We pick $p_2 > \rank \eta_2 \vee \text{max } \ran y_2$ and take $\xi_2$ to be the element $(p_2, \xi_1, m_{2j_0-1}, e_{\eta_2}^*)$, noting that this tuple is indeed in $\Delta_{p_2}$. 

Continuing in this way, we obtain a $(48, 2j_0-1, 0)$-weak dependent sequence $(Sy_i)$.  By proposition \ref{0DepSeqUpperEst} we see that \[
\| m_{2j_0-1}n_{2j_0-1}^{-1} \sum_{i=1}^{n_{2j_0-1}} Sy_i \| \le 70\times48 m_{2j_0-1}^{-1} < \frac{\delta\nu}{2} \]
with the final inequality following by the choice of $j_0$. On the other hand, \[
\sum_{i=1}^{n_{2j_0-1}} y_i (\xi_{n_{2j_0-1}}) = m_{2j_0-1}^{-1} \sum_{i=1}^{n_{2j_0-1}} y_i(\eta_i) \geq m_{2j_0-1}^{-1}n_{2j_0-1} \frac{\nu}{2}
\]
So,
\begin{align*}
\| m_{2j_0-1}n_{2j_0-1}^{-1} \sum_{i=1}^{n_{2j_0-1}} Sy_i \| &\geq \delta \|  m_{2j_0-1}n_{2j_0-1}^{-1} \sum_{i=1}^{n_{2j_0-1}} y_i \| \\
&\geq \delta m_{2j_0-1}n_{2j_0-1}^{-1} \sum_{i=1}^{n_{2j_0-1}} y_i (\xi_{n_{2j_0-1}}) \geq \frac{\delta\nu}{2}
\end{align*}
This contradiction completes the proof.
\end{proof}
We immediately obtain
\begin{cor}
The operators $S^j \colon \X_k \to \X_k$ ($j \geq 1$) are strictly singular.
\end{cor}

\section{The HI Property}
It only remains to see that the spaces $\X_k$ are hereditarily indecomposable. The proof is sufficiently close to the corresponding proof of \cite{AH} that we will omit most of the details. We first observe we have the following generalisations of lemmas 8.8 and 8.9 of \cite{AH}.
\begin{lem} \label{1DepSeqLem}
Let $(x_i)_{i\leq n_{2j_0-1}}$ be a $(C, 2j_0-1, 1)-$weak dependent sequence in $\X_k$ and let $J$ be a sub-interval of $[1, n_{2j_0-1}]$. For any $\gamma' \in \Gamma$ of weight $m_{2j_0-1}$ we have \[
\left| \sum_{i\in J} (-1)^{i}x_i(\gamma') \right| \leq 7C
\]
It follows that \[
\| n_{2j_0-1}^{-1} \sum_{i=1}^{n_{2j_0-1}} x_i \| \geq m_{2j_0-1}^{-1} \quad \text{ but } \quad \| n_{2j_0-1}^{-1}\sum_{i=1}^{n_{2j_0-1}} (-1)^i x_i \| \leq 70C m_{2j_0-1}^{-2}
\]
\end{lem}
\begin{proof}
The proof of the first claim is sufficiently close to the proof of lemma \ref{ZeroDepSeqLem} that we omit any more details. The second part of the lemma is proved in the same way as lemma 8.9 of \cite{AH}
\end{proof}

To see the spaces $\X_k$ are HI, we claim it will be enough to see that we have the following lemma
\begin{lem} \label{1WeakPair}
Let $Y$ be a block subspace of $\X_k$. There exists $\delta > 0$ such that whenever $j, p \in \N, \ve > 0$, there exists $q \in \N, x \in Y, \eta \in \Gamma$ with $\ran x \subseteq (p,q)$ and $(x, \eta)$ a $(96\delta^{-1}, 2j, 1, \ve)$ weak exact pair.
\end{lem}
We omit the proof of lemma \ref{1WeakPair}. It is essentially the same as lemma \ref{WeakExactPairs} combined with the proof of \cite{AH}, lemma 8.6. 

\begin{prop}
$\X_k$ is hereditarily indecomposable.
\end{prop}
\begin{proof}
By lemma \ref{1WeakPair}, given two block subspaces $Y$ and $Z$ of $\X_k$ there exists some $\delta > 0$ such that for all $j_0 \in \N$, we can construct $(96\delta^{-1}, n_{2j_0-1}, 1)-$weak dependent sequences, $(x_i)_{i \leq n_{2j_0-1}}$ with $x_i \in Y$ when $i$ is odd and $x_i \in Z$ when $i$ is even. Using lemma \ref{1DepSeqLem} and the same argument as in \cite{AH}, we conclude that $\X_k$ is HI as required. 
\end{proof}

\section{Concluding Remarks}

\subsection{$\mathcal{L}(\X_k)$ as a Banach algebra}

The structure of norm closed ideals in the algebra $\mathcal{L}(X)$ of all bounded linear operators on an infinite dimensional Banach space $X$ is generally not understood. It is known that for the $\ell_p$ spaces, $1 \leq p < \infty$, and $c_0$, there is only one non-trivial closed ideal in $\mathcal{L}(X)$, namely the ideal of compact operators. This was proved by Calkin, \cite{Calkin}, for $\ell_2$ and then extended to $\ell_p$ and $c_0$ by Gohberg et al., \cite{Gohberg}. More recently, the complete structure of closed ideals in $\mathcal{L}(X)$ was described in \cite{LLR04} for $X = (\oplus_{n=1}^{\infty} \ell_2^n )_{c_0}$ and in \cite{LSZ06} for $X= (\oplus_{n=1}^{\infty} \ell_2^n )_{\ell_1}$.  In both cases, there are exactly two nested proper closed  ideals. Until the space constructed by Argyros and Haydon, \cite{AH}, these were the only known separable, infinite dimensional Banach spaces for which the ideal structure of the operator algebra is completely known. 

Clearly the space $\XK$ of Argyros and Haydon provides another example of a separable Banach space for which the compact operators is the only (proper) closed ideal in the operator algebra. The spaces constructed in this paper allow us to add to the list of spaces for which the ideal structure of $\mathcal{L}(X)$ is completely known. In fact, we see that we can construct Banach spaces for which the ideals of the operator algebra form a finite, totally ordered lattice of arbitrary length. More precisely, 

\begin{lem}
There are exactly $k$ norm closed, proper ideals in $\mathcal{L}(\X_k)$ The lattice of closed ideals is given by \[
\mathcal{K}(\X_k)  \subsetneq \langle S^{k-1} \rangle \subsetneq \langle S^{k-2} \rangle \dots \langle S \rangle \subsetneq \mathcal{L}(X_k)
\]
Here, if $T$ is an operator on $\X_k$, $\langle T \rangle$ is the norm closed ideal in $\mathcal{L}(\X_k)$ generated by T.
\end{lem}

\end{document}